\documentclass[11pt]{article}
\usepackage[T1]{fontenc}
\usepackage[latin9]{luainputenc}
\usepackage{geometry}
\usepackage{color}
\usepackage{mathrsfs}
\usepackage{amsmath}
\usepackage{amsthm}
\usepackage{amssymb}
\usepackage{mathabx}
\usepackage{mathrsfs}

\usepackage{graphicx}
\usepackage{dsfont}
\usepackage{mathtools}

\usepackage{tikz}
\usetikzlibrary{decorations.pathreplacing,decorations.markings}
\usetikzlibrary{arrows}

\pgfarrowsdeclare{biggertip}{biggertip}{  
  \setlength{\arrowsize}{0.4pt}  
  \addtolength{\arrowsize}{.5\pgflinewidth}  
  \pgfarrowsrightextend{0}  
  \pgfarrowsleftextend{-5\arrowsize}  
}{  
  \setlength{\arrowsize}{0.4pt}  
  \addtolength{\arrowsize}{.5\pgflinewidth}  
  \pgfpathmoveto{\pgfpoint{-5\arrowsize}{4\arrowsize}}  
  \pgfpathlineto{\pgfpointorigin}  
  \pgfpathlineto{\pgfpoint{-5\arrowsize}{-4\arrowsize}}  
  \pgfusepathqstroke  
}

\tikzset{middlearrow/.style={
        decoration={markings,
            mark= at position 0.5 with {\arrow{#1}} ,
        },
        postaction={decorate}
    }
}

\tikzset{
  on each segment/.style={
    decorate,
    decoration={
      show path construction,
      moveto code={},
      lineto code={
        \path [#1]
        (\tikzinputsegmentfirst) -- (\tikzinputsegmentlast);
      },
      curveto code={
        \path [#1] (\tikzinputsegmentfirst)
        .. controls
        (\tikzinputsegmentsupporta) and (\tikzinputsegmentsupportb)
        ..
        (\tikzinputsegmentlast);
      },
      closepath code={
        \path [#1]
        (\tikzinputsegmentfirst) -- (\tikzinputsegmentlast);
      },
    },
  },
  mid arrow/.style={postaction={decorate,decoration={
        markings,
        mark=at position .5 with {\arrow[#1]{stealth}}
      }}},
}

\usepackage{float}

\usepackage[unicode=true,
 bookmarks=true,bookmarksnumbered=true,bookmarksopen=false,
 breaklinks=false,pdfborder={0 0 1},backref=false,colorlinks=true]
 {hyperref}
\hypersetup{
 linkcolor=blue, citecolor=black, urlcolor=blue, filecolor=blue, pdfpagelayout=OneColumn, pdfnewwindow=true, pdfstartview=XYZ, plainpages=false}

\makeatletter
\numberwithin{equation}{section}
\theoremstyle{plain}
\newtheorem{thm}{\protect\theoremname}[section]
  \theoremstyle{remark}
  \newtheorem{rem}[thm]{\protect\remarkname}
  \theoremstyle{definition}
  \newtheorem{defn}[thm]{\protect\definitionname}
  \theoremstyle{plain}
  \newtheorem{lem}[thm]{\protect\lemmaname}
  \theoremstyle{plain}
  \newtheorem{prop}[thm]{\protect\propositionname}
  \theoremstyle{remark}
  \newtheorem*{rem*}{\protect\remarkname}

\usepackage{ifpdf} 
\ifpdf 

 \IfFileExists{lmodern.sty}{\usepackage{lmodern}}{}

\fi 

\usepackage[figure]{hypcap}
\usepackage{bbm}

\let\myTOC\tableofcontents
\renewcommand\tableofcontents{%
  \frontmatter
  \pdfbookmark[1]{\contentsname}{}
  \myTOC
  \mainmatter }

\usepackage{fancyhdr}

\@ifundefined{extrarowheight}
 {\usepackage{array}}{}
\date{ }

\def\definitionname{Definition}
\def\theoremname{Theorem}
\def\propositionname{Proposition}
\def\lemmaname{Lemma}
\def\corollaryname{Corollary}
\def\remarkname{Remark}

\makeatother

  \providecommand{\definitionname}{Definition}
  \providecommand{\lemmaname}{Lemma}
  \providecommand{\propositionname}{Proposition}
  \providecommand{\remarkname}{Remark}
\providecommand{\theoremname}{Theorem}
\newcommand{\N}{\mathbb N}
\newcommand{\R}{\mathbb R}
\newcommand{\cA}{\mathcal A}
\newcommand{\cB}{\mathscr{B}}
\newcommand{\cL}{\mathcal L}
\newcommand{\cV}{\mathcal V}

\newcommand{\ccV}{\mathscr V}
\def\ds{\displaystyle}

\begin{document}
\date{\today}

\title{Finite Horizon Mean Field Games on Networks}

\author{Yves Achdou \thanks { Univ. Paris Diderot, Sorbonne Paris Cit{\'e}, Laboratoire Jacques-Louis Lions, UMR 7598, UPMC, CNRS, F-75205 Paris, France.
 achdou@ljll.univ-paris-diderot.fr},
Manh-Khang Dao \thanks {Department of Mathematics, KTH Royal Institute of Technology, Sweden},
Olivier Ley \thanks {Univ Rennes, INSA Rennes, CNRS, IRMAR - UMR 6625, F-35000 Rennes, France. olivier.ley@insa-rennes.fr},
Nicoletta Tchou \thanks {Univ Rennes, CNRS, IRMAR - UMR 6625, F-35000 Rennes, France. nicoletta.tchou@univ-rennes1.fr}
}
\maketitle

\begin{abstract} We consider finite horizon stochastic mean field games in  which the state space is a network.
They are described by a system coupling  a backward in time Hamilton-Jacobi-Bellman equation and a forward in time 
Fokker-Planck equation.
The  value function $u$ is continuous and satisfies general Kirchhoff conditions at the vertices. 
The density $m$ of the distribution of states satisfies dual transmission conditions:
 in particular, $m$ is generally discontinuous across the vertices, and the values of $m$ on each side of the vertices satisfy special compatibility conditions. 
The stress is put on the case when the Hamiltonian is Lipschitz continuous. Existence and uniqueness are proven.
\end{abstract}
\section{\label{sec: Mean-field-games}Introduction and main results}

This work is the continuation of \cite{ADLT} which was devoted to mean field games on networks in the case of an infinite time horizon. The topic of  mean field games (MFGs for short) is more and more investigated since the pioneering works \cite{LL2006-A,MR2271747,LL2007} of  Lasry and Lions: it aims at studying the  asymptotic behavior of  stochastic differential games (Nash equilibria) as the number $N$ of agents tends to infinity. We refer to  \cite{ADLT} for a more extended discussion on MFGs and for additional references on the analysis of the system of PDEs that stem from the model when there is no common noise.

A network (or a graph) is a set of items, referred to as vertices (or nodes or crosspoints), with connections between them referred to as edges. In the recent years, there has been an increasing interest  in the investigation of dynamical systems and differential equations on networks, in particular in connection with problems of data transmission and traffic management. The literature  on optimal control in which the state variable takes its values on a network is recent:  deterministic control problems and related Hamilton-Jacobi equations were studied in   \cite{MR3057137,MR3358634,MR3023064,MR3621434,MR3556345,MR3729588}. Stochastic processes on networks and related Kirchhoff conditions at the vertices  were studied in \cite{FS2000,FW1993}.

 The present work is devoted to  finite horizon stochastic mean field games (MFGs) taking place on networks. 
The most important difficulty will be to deal with the transition conditions at the vertices. The latter  are obtained 
 from the theory of  stochastic control in \cite{FW1993,FS2000}, see Section \ref{subsec: A derivation of the MFG system} below.
In \cite{CM2016}, the first article on MFGs on networks, Camilli and Marchi consider a particular type of Kirchhoff condition  at the vertices for the value function: this condition comes from an assumption which can be informally stated as follows: consider a vertex $\nu$ of the network and assume that it is the intersection of $p$ edges $\Gamma_{1},\dots, \Gamma_{p} $, ;
 if, at time $\tau$, the controlled stochastic process $X_t$ associated to a given agent hits $\nu$, 
then the probability that $X_{\tau^+}$ belongs to $\Gamma_i$ is proportional to the diffusion coefficient in $\Gamma_i$. Under this assumption, it can be seen that the density of the distribution of states is continuous at the vertices of the network.  
In the present work, the above mentioned  assumption is not made any longer.  Therefore, it will be seen below that the value function  satisfies more general Kirchhoff conditions, and accordingly, that the density of the distribution of states is no longer continuous at the vertices; the continuity condition is then replaced by suitable
 compatibility conditions on the jumps across the vertices. A complete study of the system of differential equations arising in 
infinite horizon mean field games on networks 
with at most quadratic Hamiltonians and very general coupling costs has been supplied in a previous work, see \cite{ADLT}.

In the present work, we focus on a more basic case, namely  finite horizon MFG with globally Lipschitz Hamiltonian 
with  rather strong assumptions on the coupling cost. This will allow us to concentrate on the difficulties induced by the Kirchhoff conditions.  Therefore, this work should be seen as a first and necessary step 
in order to deal with more difficult situations, for example with quadratic or subquadratic Hamiltonians. We believe that 
treating such cases will be possible by combining the results contained in the present work with methods that can be found in \cite{Parma,Pumi}, see also \cite{BMP,MR0241822,MR1465184} for references on Hamilton-Jacobi equations.

After obtaining the  transmission conditions at the vertices for both the value function and the density,
we shall prove existence and  uniqueness of weak solutions  of the uncoupled Hamilton-Jacobi-Bellman (HJB) and Fokker-Planck (FP) equations (in suitable space-time Sobolev spaces), and regularity results.

The present work is organized as follows: the remainder of Section \ref{sec: Mean-field-games} is devoted to
setting the problem and obtaining the system of partial differential equations and the transmission conditions at the vertices.
 Section~\ref{subsec: The-linear-parabolic} contains useful results on a modified heat equation in the network with general Kirchhoff conditions. Section~\ref{sec:fokk-planck-equat} is devoted to  the Fokker-Planck equation.
Weak solutions are defined by using  a special pair of Sobolev spaces of functions defined on the network referred to as $V$ and $W$ below. Section \ref{sec:hamilt-jacobi-equat} is devoted to 
the HJB equation supplemented with the Kirchhoff conditions: it addresses the main difficulty of the work, consisting of  obtaining regularity results for the weak solution 
(note that, to the best of our knowledge, such results for networks and general Kirchhoff conditions do not exist in the literature).
 Finally,  the proofs of the main results of existence and uniqueness for the MFG system of partial differential equations are completed in Section~\ref{sec:exist-uniq-regul}.

\subsection{Networks and function spaces}

\subsubsection{The geometry}\label{sec:geometry}

A bounded network $\Gamma$ (or a bounded connected graph) is a connected subset of $\R ^n$ made of a finite number of  bounded non-intersecting straight segments, referred to as edges, which connect  
nodes referred to as vertices. The finite collection of vertices and the finite set of closed edges are respectively denoted by $\mathcal{V}:=\left\{ \nu_{i}, i\in I\right\}$
and  $\mathcal{E}:=\left\{ \Gamma_{\alpha}, \alpha\in\mathcal{A}\right\}$, 
where $I$ and $A$ are finite sets of indices contained in $\N$.  We assume that for $\alpha,\beta \in \cA$,  if $\alpha\not=\beta$, 
 then $\Gamma_\alpha\cap \Gamma_\beta$ is either empty or made of  a single vertex.
 The length of $\Gamma_{\alpha}$ is denoted by $\ell_{\alpha}$.
Given $\nu_{i}\in\mathcal{V}$, the set of indices of edges that are
adjacent to the vertex $\nu_{i}$ is denoted by
 $\mathcal{A}_{i}=\left\{ \alpha\in\mathcal{A}:\nu_{i}\in\Gamma_{\alpha}\right\} $.
A vertex $\nu_{i}$ is named a {\sl boundary vertex} if $\sharp\left(\mathcal{A}_{i}\right)=1$,
otherwise it is named a {\sl transition vertex}. The set containing all
the boundary vertices is named the {\sl boundary} of the network and is 
denoted by $\partial\Gamma$ hereafter.

The edges $\Gamma_{\alpha}\in\mathcal{E}$  are oriented in an arbitrary manner. In most of what follows, we shall make the 
following arbitrary choice that an edge $\Gamma_{\alpha}\in\mathcal{E}$ connecting two vertices $\nu_i$ and $\nu_j$, with $i<j$ is oriented from $\nu_i$ toward $\nu_j$:
this induces a natural parametrization $\pi_\alpha: [0,\ell_\alpha]\to \Gamma_\alpha=[\nu_i,\nu_j]$:
\begin{equation}
\pi_\alpha(y)=(\ell_\alpha-y)\nu_i+y\nu_j\quad\text{for } y\in  [0,\ell_\alpha]. \label{eq: parametrization}
\end{equation}
For a function 
$v:\Gamma\rightarrow\mathbb{R}$ and $\alpha\in\mathcal{A}$,
we define $v_\alpha: (0,\ell_\alpha)\rightarrow\mathbb{R}$ by
\begin{displaymath}
v_{\alpha} (y):=v\circ\pi_{\alpha} (y),\quad \hbox{ for all }y\in (0,\ell_\alpha).  
\end{displaymath}
The function $v_\alpha$ is a priori defined only in $(0,\ell_\alpha)$.
When it is possible, we extend it by continuity at the boundary by setting
\begin{eqnarray*}
\displaystyle v_{\alpha}\left(0\right):=\lim_{y\rightarrow0^{+}}v_{\alpha}\left(y\right)
\text{ and }
v_{\alpha}\left(\ell_{\alpha}\right):=\lim_{y\rightarrow\ell_{\alpha}^{-}}v_{\alpha}\left(y\right).
\end{eqnarray*}
In that latter case, we can define
\begin{equation}
v|_{\Gamma_{\alpha}}\left(x\right)=\begin{cases}
v_{\alpha}\left(\pi_{\alpha}^{-1}\left(x\right)\right), & \text{if }x\in\Gamma_{\alpha}\backslash\mathcal{V},\\
\displaystyle v_{\alpha}\left(0\right)=\lim_{y\rightarrow0^{+}}v_{\alpha}\left(y\right), & \text{if }x=\nu_{i},\\
\displaystyle v_{\alpha}\left(\ell_{\alpha}\right)=\lim_{y\rightarrow\ell_{\alpha}^{-}}v_{\alpha}\left(y\right), & \text{if }x=\nu_{j}.
\end{cases}\label{eq: v at the vertices}
\end{equation}
Notice that $v|_{\Gamma_{\alpha}}$ does not coincide with the original function $v$ at the vertices in general
when $v$ is not continuous.

\begin{rem}\label{rem:new network}
In what precedes, the edges have been arbitrarily  oriented from the vertex with the smaller index toward the vertex with the larger one.
 Other choices are of course possible.
 In particular, by possibly dividing a single edge into two, adding thereby new artificial vertices, it is always possible to assume that for all vertices $\nu_i\in\mathcal{V}$,
\begin{equation}
  \text{either }\pi_{\alpha}(0)=\nu_i,\text{ for all }\alpha\in\mathcal{A}_i
  \text{ or }\pi_{\alpha}(\ell_\alpha)=\nu_i,\text{ for all }\alpha\in\mathcal{A}_i.\label{oriented networks}
\end{equation} 
This idea was used by  Von Below in \cite{Below1988}: some edges of $\Gamma$ are cut into two by adding artificial vertices
so that the new oriented network $\overline{\Gamma}$ has the property \eqref{oriented networks}, see Figure \ref{fig: Network intro} for an example.

\begin{figure}[H]
  \begin{center}
    \begin{tikzpicture}[scale=0.5, trans/.style={thick,<->,shorten >=2pt,shorten <=2pt,>=stealth} ]
      \draw (0.8,5.5)node[left] {$\Gamma_1$};
      \draw[middlearrow={triangle 90}] (-2,5) node [left]{$\nu_1$} -- (5,5)  node [right]{$\nu_2$};
      \draw (-2,2.5)node[left] {$\Gamma_2$};
      \draw[middlearrow={triangle 90}] (-2,5) -- (-2,0) node [left]{$\nu_3$};
      \draw (0.8,3.3)node[right] {$\Gamma_3$};
      \draw[middlearrow={triangle 90}] (-2,5) -- (5,0) node [right]{$\nu_4$};
      \draw (0.8,-0.5)node[left] {$\Gamma_4$};
      \draw[middlearrow={triangle 90}] (-2,0) -- (5,0);

      \draw (12.8,5.5)node[left] {$\tilde\Gamma_1$};
      \draw[middlearrow={triangle 90}] (10,5) node [left]{$\tilde \nu_1$} -- (17,5) node [right]{$\tilde\nu_2$};
      \draw (10,3.75)node[left] {$\tilde \Gamma_2$};
      \draw[-{triangle 90}] (10,5) -- (10,2.5) node [left]{$\tilde \nu_5$};
       \draw (10,1.25)node[left] {$\tilde \Gamma_5$};
      \draw[-{triangle 90}] (10,0)  node [left]{$\tilde \nu_3$} -- (10,2.5) ;
      \draw (12.8,3.3)node[right] {$\tilde\Gamma_3$};
      \draw[middlearrow={triangle 90}] (10,5) -- (17,0) node [right]{$\tilde \nu_4$};
      \draw (12.8,-0.5)node[left] {$\tilde\Gamma_4$};
      \draw[middlearrow={triangle 90}] (10,0) -- (17,0);
    \end{tikzpicture}
    \caption{Left: the network $\Gamma$ in which the edges are oriented  toward the vertex with larger index ($4$ vertices and $4$ edges). Right: a new network $\tilde \Gamma$ obtained by adding an artificial vertex ($5$ vertices and $5$ edges): the oriented edges sharing a given vertex $\nu$ either have all their starting point equal $\nu$, or have all  their terminal point equal $\nu$.}
   \label{fig: Network intro}
  \end{center}
\end{figure}
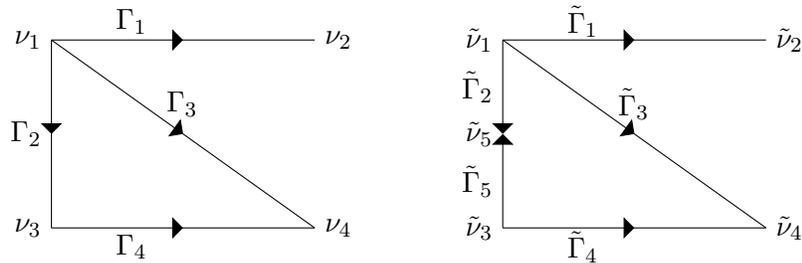

\end{rem}

\subsubsection{Function spaces related to the space variable}
\label{sec:function-spaces}

The set of continuous functions on $\Gamma$ is denoted by $C(\Gamma)$
and we set
\[
\begin{array}[c]{ll}
&PC\left(\Gamma\right) \\ =& \left\{ v : \Gamma \to \R \;:  \hbox{  for all $\alpha\in\mathcal{A}$}, \left|
    \begin{array}[c]{l}
\hbox{$v_\alpha\in C(0,\ell_\alpha)$ }\\
  \hbox{ $v_\alpha$ can be extended by continuity to $[0,\ell_\alpha]$.}     
    \end{array} \right.
\right\}.  
\end{array}
\]
By the definition of piecewise continuous functions $v\in PC(\Gamma)$, for all $\alpha\in \mathcal{A}$, it is possible to define $v|_{\Gamma_\alpha}$ by~\eqref{eq: v at the vertices} and we have $v|_{\Gamma_\alpha}\in C(\Gamma_\alpha)$,
$v_\alpha\in C([0,\ell_{\alpha}])$.

For $m\in\mathbb{N}$, the space of $m$-times continuously differentiable functions on $\Gamma$ is defined by
\[
C^{m}\left(\Gamma\right):=\left\{ v\in C\left(\Gamma\right):v_{\alpha}\in C^{m}\left(\left[0,\ell_{\alpha}\right]\right)\text{ for all }\alpha\in\mathcal{A}\right\}.
\]
Notice that $v\in C^{m}\left(\Gamma\right)$ is assumed to be continuous on $\Gamma$,
and that its restriction $ v_{|\Gamma_\alpha}$ to each edge $\Gamma_\alpha$ belongs 
to $C^m(\Gamma_\alpha)$. 
The space $C^{m}\left(\Gamma\right)$ is endowed with the norm
$ 
\left\Vert v\right\Vert _{C^{m}\left(\Gamma\right)}:= {\sum}_{\alpha\in\mathcal{A}}{\sum}_{k\le m}\left\Vert \partial^{k}v_{\alpha}\right\Vert _{L^{\infty}\left(0,\ell_{\alpha}\right)}
$.
For $\sigma\in\left(0,1\right)$, the space $C^{m,\sigma}\left(\Gamma\right)$, 
contains the functions
$v\in C^{m}\left(\Gamma\right)$ such that $\partial^{m}v_{\alpha}\in C^{0,\sigma}\left(\left[0,\ell_{\alpha}\right]\right)$
for all $\alpha\in \mathcal{A}$; it is endowed with the norm
$$\displaystyle{
\left\Vert v\right\Vert _{C^{m,\sigma}\left(\Gamma\right)}:=\left\Vert v\right\Vert _{C^{m}\left(\Gamma\right)}+\sup_{\alpha\in\mathcal{A}}\sup_{y\ne z \atop y,z\in\left[0,\ell_{\alpha}\right]}\dfrac{\left|\partial^{m}v_{\alpha}\left(y\right)-\partial^{m}v_{\alpha}\left(z\right)\right|}{\left|y-z\right|^{\sigma}}}.$$

For a positive integer $m$ and a function $v\in C^{m}\left(\Gamma\right)$,
we set for $k\le m$,
\begin{equation}
\partial^{k}v\left(x\right)=\partial^{k}v_{\alpha}\left(\pi_{\alpha}^{-1}\left(x\right)\right)\text{ if }x\in\Gamma_{\alpha}\backslash{\mathcal{V}}.\label{eq: derivative}
\end{equation}
For a vertex $\nu$, we define $\partial_{\alpha}v\left(\nu\right)$ as the {\sl outward} directional
derivative of $v|_{\Gamma_{\alpha}}$ at $\nu$ as follows:
\begin{equation}
\partial_{\alpha}v\left(\nu\right):=\begin{cases}
{\displaystyle \lim_{h\rightarrow0^{+}}
\dfrac{v_{\alpha}\left(0\right)-v_{\alpha}\left(h\right)}{h}}, & \text{if }\nu=\pi_{\alpha}\left(0\right),\\
{\displaystyle \lim_{h\rightarrow0^{+}}
\dfrac{v_{\alpha}\left(\ell_{\alpha}\right)-v_{\alpha}\left(\ell_{\alpha}-h\right)}{h}}, & \text{if }\nu=\pi_{\alpha}\left(\ell_{\alpha}\right).
\end{cases}\label{eq: inward derivative}
\end{equation}
For all $i\in I$ and $\alpha\in \cA_i$, setting 
\begin{equation}
  \label{eq:1}
 n_{i\alpha}=\left\{
 \begin{array}[c]{rl}
 1 & \text{if }   \nu_i =    \pi_{\alpha} (\ell_\alpha),\\
 -1 & \text{if }  \nu_i =    \pi_{\alpha} (0),   
 \end{array}\right.
\end{equation}
we have
\begin{equation}
\label{eq:2}
  \partial_\alpha  v(\nu_i)= n_{i\alpha}\,\partial v|_{\Gamma_{\alpha}}(\nu_i)= n_{i\alpha} \, \partial v_ \alpha (\pi^{-1}_\alpha (\nu_i)) .
\end{equation}

\begin{rem}\label{sec:netw-funct-spac}
Changing the orientation of the edge does not change the value of $\partial_\alpha v(\nu)$ in (\ref{eq: inward derivative}). 
\end{rem}

We say that $v$ is Lebesgue-integrable on $\Gamma_{\alpha}$ if $v_{\alpha}$ is Lebesgue-integrable on $(0,\ell_{\alpha})$.
In this case, for all $x_1,x_2\in \Gamma_{\alpha}$,
\begin{eqnarray}\label{int-segment}
\int_{[x_1,x_2]}v\left(x\right)dx:=\int_{\pi_\alpha^{-1}(x_1)}^{\pi_\alpha^{-1}(x_2)} v_{\alpha}\left(y\right)dy.
\end{eqnarray}
When $v$  is Lebesgue-integrable on $\Gamma_{\alpha}$ for all $\alpha\in\mathcal{A}$, we say that
$v$ is Lebesgue-integrable on $\Gamma$ and we define
\begin{eqnarray*}
\int_{\Gamma}v\left(x\right)dx:=\sum_{\alpha\in\mathcal{A}}\int_{0}^{\ell_{\alpha}}v_{\alpha}\left(y\right)dy.
\end{eqnarray*}
The space
$L^{p}\left(\Gamma\right)  =\left\{ v:v|_{\Gamma_{\alpha}}\in L^{p}\left(\Gamma_{\alpha}\right)\text{ for all \ensuremath{\alpha\in\mathcal{A}}}\right\}$,  $p\in [1,\infty]$, 
is endowed with the norm $\left\Vert v\right\Vert _{L^{p}\left(\Gamma\right)}:=  \left( \sum_{\alpha\in\mathcal{A}}\left\Vert v_{\alpha}\right\Vert _{L^{p} \left(0,\ell_{\alpha}\right)}^p \right)^{\frac 1 p}$
if $1\le p<\infty$, and $\max_{\alpha\in \cA} \|v_\alpha\|_{L^\infty\left(0,\ell_{\alpha}\right)}$ if $p=+\infty$.
We shall also need to deal with functions on $\Gamma$ whose restrictions to the edges are  weakly-differentiable:
we shall use the same notations for the weak derivatives.

\begin{defn} \label{def: Sobolev space}
For any integer $s\ge 1$ and any real number $p\ge 1$, the Sobolev space $W^{s,p}_{b}(\Gamma)$ is defined as follows
\[
W^{s,p}_{b}(\Gamma):=\left\{ v:\Gamma\rightarrow\R \text{ s.t. }v_{\alpha}\in W^{s,p}\left(0,\ell_{\alpha}\right)\text{ for all }\alpha\in\mathcal{A}\right\},
\]
and endowed with the norm
\[
\left\Vert v\right\Vert _{W^{s,p}_{b}\left(\Gamma\right)}=\left(\sum^{s}_{k=1}\sum_{\alpha\in\mathcal{A}}
\left\Vert \partial^{k}v_{\alpha}\right\Vert _{L^{p}\left(0,\ell_{\alpha}\right)}^{p}+
\left\Vert v\right\Vert _{L^p(\Gamma)}^{p}\right)^{\frac 1 p}.
\]
For $s\in \N \backslash \{0\}$,  we also set $H^s_b(\Gamma)= W^{s,2}_{b}(\Gamma)$ and
 $H^s(\Gamma)=C(\Gamma)\cap H^s_b(\Gamma)$.
\end{defn}

Finally, when dealing with probability distributions in mean field games, we will often use 
the set $\mathcal{M}$ of  probability densities, i.e., $m\in L^1(\Gamma)$, $m\geq 0$ and $\int_\Gamma m(x)dx=1$.

\subsubsection{Some space-time function spaces}
\label{sec:function-space-time}
The space of continuous real valued functions on $\Gamma\times [0,T]$ is denoted by $C(\Gamma\times [0,T])$.

Let $PC(\Gamma\times [0,T])$  be the space of  the functions $v:\Gamma\times [0,T]\to \R$ such that
\begin{enumerate}
\item for all $t\in [0,T]$, $v(\cdot,t)$ belongs to $PC(\Gamma)$
\item for all $\alpha\in \cA$, $v|_{\Gamma_\alpha\times [0,T]}$ is continuous on $\Gamma_\alpha\times [0,T]$;
\end{enumerate}
For a function $v\in PC(\Gamma\times [0,T])$, $\alpha\in \cA$, we set $v_\alpha(y,t)= v|_{\Gamma_\alpha\times[0,t]}(\pi_\alpha(y), t)$
for all $(y,t)\in[0,\ell_\alpha]\times [0,T]$.

For two nonnegative integers $m$ and $n$, let $C^{m,n}(\Gamma\times [0,T])$ be the space of continuous real valued functions $v$ on $\Gamma\times [0,T]$ such that for all $\alpha\in \cA$,  $v|_{\Gamma_\alpha\times [0,T]} \in  C^{m,n}(\Gamma_\alpha\times [0,T])$. For $\sigma\in (0,1)$, $\tau\in (0,1)$, we define in the same manner $C^{m+\sigma,n+\tau}(\Gamma\times [0,T])$


Useful results on continuous and compact embeddings of space-time function spaces are given in  Appendix~\ref{app:embeddings}.

\subsection{A class of stochastic processes on $\Gamma$}
\label{sec:class-stoch-proc}

After rescaling the edges, it may be assumed that $\ell_{\alpha}=1$ for all $\alpha\in\mathcal{A}$. 
Let $ \mu_{\alpha} , \alpha\in\mathcal{A}$ and  $ p_{i\alpha}, i\in I,\alpha\in\mathcal{A}_{i}$
be positive constants such that $\sum_{\alpha\in\mathcal{A}_{i}}p_{i\alpha}=1$. Consider also a real valued function $a\in PC(\Gamma\times[0,T])$, such that, for all $\alpha\in \cA$, $t\in [0,T]$,
 $a|_{\Gamma_\alpha}(\cdot, t)$ belongs to $C^1(\Gamma_\alpha)$.

As in Remark~\ref{rem:new network}, we make the assumption (\ref{oriented networks}) by possibly adding artificial nodes: if $\nu_i$ is such an artificial node, then $\sharp(\cA_i)=2$, and we assume that $p_{i\alpha}=1/2$ for $\alpha\in \cA_i$. The diffusion parameter $\mu$ has the same value on the two sides of an artificial vertex. Similarly, the function $a$ 
does not have jumps across an artificial vertex.

 Consider a Brownian motion $(W_t)$ defined on the real line. 
Following Freidlin and Sheu  (\cite{FS2000}), we know that there exists a unique Markov process on $\Gamma$ with continuous sample paths that can be written $(X_t, \alpha_t)$ where  $X_t\in \Gamma_{\alpha_t}$ (if $X_t=\nu_i$, $i\in I$, $\alpha_t$ is arbitrarily chosen as the smallest index in $\cA_i$)
 such that, defining the process  $x_t= \pi_{\alpha_t}(X_t)$  with values in $[0,1]$, 
 \begin{itemize}
 \item we have
   \begin{equation}
     \label{eq:3}
    dx_t= \sqrt{2\mu_{\alpha_t}} dW_t + a_{\alpha_t}(x_t,t) dt  + d\ell_{i,t} +d h_{i,t},
   \end{equation}
 \item $\ell_{i,t}$ is continuous non-decreasing process (measurable with respect to the $\sigma$-field generated by $(X_t,\alpha_t)$) which increases only when $X_t=\nu_i$ and $x_t=0$,
\item $h_{i,t}$ is  continuous non-increasing process (measurable with respect to the $\sigma$-field generated by $(X_t,\alpha_t)$) which decreases only when $X_t=\nu_i$ and $x_t=1$,
 \end{itemize}
 and for all function $v\in C^{2,1}(\Gamma\times [0,T])$ such that 
\begin{equation}
\label{eq:4}
  \sum_{\alpha\in\mathcal{A}_{i}}p_{i\alpha}\partial_{\alpha}v\left(\nu_{i},t\right)=0,\; \hbox{ for all }i\in I,t\in [0,T], 
\end{equation}
the process
\begin{equation}\label{eq:5}
  M_t=v(X_t,t)-\int_0^t \Bigl(\partial_t  v\left(X_s,s\right)
 +\mu_{\alpha_s}\partial^{2}v\left(X_s,s\right)+a|_{\Gamma_{\alpha_s}}\left(X_s,s\right)\partial v\left(X_s,s\right) \Bigr) ds
\end{equation}
is a martingale, i.e.,
\begin{equation}
  \label{eq:6}
\mathbb{E}(M_t| X_s)= M_s, \quad\hbox{for all } 0\le s<t\le T.
\end{equation}
For  what follows, it will be convenient to set 
\begin{equation}
 D:= \left\{ u\in C^{2}\left(\Gamma\right):\sum_{\alpha\in\mathcal{A}_{i}}p_{i\alpha}\partial_{\alpha}u\left(\nu_{i}\right)=0,\; \hbox{ for all }i\in I \right\}.\label{Kirchhoff condition}
\end{equation}

\begin{rem}\label{sec:class-stoch-proc-2}
  Note that  in (\ref{eq:4}),  the condition at boundary vertices boils down to a Neumann condition.
\end{rem}

\begin{rem}
  \label{sec:class-stoch-proc-1}
The assumption that all  the edges have unit length is not restrictive, because we can always rescale the constants $\mu_\alpha$ and the piecewise continuous function $a$. 
\end{rem}

The goal is to derive the boundary value problem satisfied by
the law of the stochastic process $X_t$. Since the derivation here
is formal, we assume that 
the law of the stochastic process $X_t$ is a measure
which is absolutely continuous with respect to the Lebesgue measure on $\Gamma$ and regular
enough so that the following computations make sense.
Let $m(x,t)$ be its density. We have
\begin{equation}
   \label{eq:11}
\mathbb{E}\left[v\left(X_t,t\right)\right]=\int_{\Gamma}v\left(x,t\right)m\left(x,t\right)dx, \quad  \hbox{ for all } v\in PC(\Gamma\times [0,T]).
 \end{equation}
Consider $u\in C^{2,1}(\Gamma\times [0,T])$ such that for all $t\in [0,T]$, $u(\cdot, t)\in D$. Then, from (\ref{eq:5})-(\ref{eq:6}), we see that 
\begin{equation}
  \label{eq:12}
  \mathbb{E}\left[u\left(X_t,t\right)\right]=
  \mathbb{E}\left[u\left(X_0,0\right)\right]
+ \mathbb{E}\left[ \int_0^t \Bigl(\partial_t  u\left(X_s,s\right)
 +\mu_{\alpha_s}\partial^{2}u\left(X_s,s\right)+a|_{\Gamma_{\alpha_s}}\left(X_s,s\right)\partial u\left(X_s,s\right) \Bigr) ds
\right].
\end{equation}
Using~\eqref{eq:11} and taking the time-derivative of each member of (\ref{eq:12}), we obtain
\begin{eqnarray*}
\int_\Gamma \partial_t(um)(x,t)dx =\mathbb{E}\Bigl(\partial_t  u\left(X_t,t\right)
 +\mu_{\alpha_s}\partial^{2}u\left(X_t,t\right)+a|_{\Gamma_{\alpha_s}}\left(X_t,t\right)\partial u\left(X_t,t\right) \Bigr).
\end{eqnarray*}
Using again~\eqref{eq:11}, we get
\begin{eqnarray}\label{eq:13}
 \int_\Gamma
  \left(   \mu \partial^2 u(x,t) + a(x,t)\partial u(x,t)\right)m(x,t)dx
= \int_\Gamma u(x,t)\partial_tm(x,t)dx.
\end{eqnarray}
By integration by parts, recalling~\eqref{oriented networks}, we get
\begin{eqnarray}
0 &=& \sum_{\alpha\in \mathcal{A}}  \int_{\Gamma_\alpha}
  \left( \partial_tm(x,t)-  \mu_\alpha \partial^2 m(x,t) + \partial (am)(x,t)\right)u(x,t)dx\nonumber\\
 & &- \sum_{i\in I}\sum_{\alpha\in\mathcal{A}_{i}} \left[n_{i\alpha} a|_{\Gamma_\alpha}(\nu_i,t)  m|_{\Gamma_\alpha}(\nu_i,t)
    - \mu_\alpha  \partial_\alpha m(\nu_i,t) \right]  u|_{\Gamma_\alpha}(\nu_i,t)  \nonumber\\
  & &-  \sum_{i\in I}\sum_{\alpha\in\mathcal{A}_{i}} \mu_\alpha m|_{\Gamma_\alpha}(\nu_i,t)  \partial_\alpha u(\nu_i,t),
  \label{form283}
\end{eqnarray}
where $n_{i\alpha}$ is defined in~\eqref{eq:1}.

We choose first, for every  $\alpha \in \cA$, a smooth function $u$ which is
compactly supported in $(\Gamma_\alpha \backslash \cV)\times [0,T] $. Hence $u|_{\Gamma_\beta}(\nu_i,t)=0$
and  $\partial_\beta u(\nu_i,t)=0$ for all $i\in I, \beta\in \mathcal{A}_{i}$.
Notice that $u(\cdot,t)\in D$.
It follows that $m$ satisfies 
   \begin{equation}
     \label{eq:14}
     \left(\partial_t m-\mu_{\alpha}\partial^{2}m +\partial\left(m a\right) \right)(x,t)=0, \quad
     \text{for $x\in \Gamma_\alpha \backslash \cV$, $t\in (0,T)$, $\alpha\in \cA$.}
   \end{equation}

For a smooth function $\chi: [0,T]\to \R$ compactly supported in $(0,T)$, we may choose for every $i\in I$, a smooth function $u$ such that 
 $u(\nu_j,t)=\chi(t) \delta_{i,j}$ for all $t\in [0,T]$, $j\in I$ and
 $\partial_\alpha  u(\nu_j,t)=0$ for all $t\in [0,T]$, $j\in I$ and $\alpha\in \cA_j$, we infer
 a condition for $m$ at the vertices,
 \begin{eqnarray}
 \sum_{\alpha\in\mathcal{A}_{i}} n_{i\alpha} a|_{\Gamma_\alpha}(\nu_i,t)  m|_{\Gamma_\alpha}(\nu_i,t)
 - \mu_\alpha  \partial_\alpha m(\nu_i,t)= 0
 \qquad \text{for all $i\in I$,  $t\in (0,T)$.}
 \end{eqnarray}
 This condition is called a transmission condition if $\nu_i$ is a transition
 vertex and reduces to a Robin boundary condition when $\nu_i$ is a boundary
 vertex.

 Finally, for a smooth function $\chi: [0,T]\to \R$ compactly supported in $(0,T)$,
 for every transition vertex $\nu_i\in \mathcal{V}\setminus \partial\Gamma$ and $\alpha, \beta \in \cA_i$, we choose
 $u$ such that  
 \begin{itemize}
 \item $u(\cdot,t)\in D$
 \item $\partial_\alpha  u(\nu_i,t)= \chi(t)/p_{i\alpha}$,  $\partial_\beta  u(\nu_i)= -\chi(t)/p_{i\beta}$, $\partial_\gamma  u(\nu_i)=0$ if $\gamma\in \cA_i \backslash\{\alpha, \beta\}$
 \item The directional derivatives of $u$ at the vertices $\nu\not= \nu_i$ are $0$.
 \end{itemize}
Using such a test-function  in~\eqref{form283} yields a jump condition for $m$,
\begin{eqnarray*}
\dfrac{m|_{\Gamma_{\alpha}}\left(\nu_{i},t\right)}{\gamma_{i\alpha}}=\dfrac{m|_{\Gamma_{\beta}}\left(\nu_{i},t\right)}{\gamma_{i\beta}},\quad\text{for all }\alpha,\beta\in\mathcal{A}_{i},\nu_{i}\in\mathcal{V}, t\in (0,T),
\end{eqnarray*}
in which
\begin{equation}
  \label{eq:15} \gamma_{i\alpha}= \frac {p_{i\alpha}}{\mu_\alpha},  \quad \hbox{for all } i\in I,   \alpha\in \cA_i. 
\end{equation}

Summarizing, we get the following boundary value problem  for $m$ (recall that the coefficients $n_{i\alpha}$ are defined in~\eqref{eq:1}):
\begin{equation}
  \label{eq:16}
\left\{  \begin{split}
\partial_t m-\mu_{\alpha}\partial^{2}m+\partial\left(m a\right)=0,\quad \quad  & (x,t)\in\left(\Gamma_{\alpha}\backslash\mathcal{V}\right)\times (0,T),\,\alpha\in\mathcal{A},\\
\sum_{\alpha\in\mathcal{A}_{i}}\mu_{\alpha}\partial_{\alpha}m\left(\nu_{i},t\right)
-n_{i\alpha} a|_{\Gamma_{\alpha}} (\nu_i) m|_{\Gamma_{\alpha}}\left(\nu_{i},t\right)=0,\quad \quad  & t\in (0,T),\nu_{i}\in\mathcal{V},\\
\dfrac{m|_{\Gamma_{\alpha}}\left(\nu_{i},t\right)}{\gamma_{i\alpha}}=\dfrac{m|_{\Gamma_{\beta}}\left(\nu_{i},t\right)}{\gamma_{i\beta}},\quad \quad  &t\in (0,T),~\alpha,\beta\in\mathcal{A}_{i},~\nu_{i}\in\mathcal{V},\\
m(x,0)=m_0(x), \quad \quad & x\in \Gamma.   
  \end{split}\right.
\end{equation}

\subsection{Formal derivation of the MFG system on $\Gamma$}\label{subsec: A derivation of the MFG system}

Here we aim at obtaining the MFG system of forward-backward partial differential equations on the network, at least formally.
The assumptions that we are going to make below on the optimal control problem are a little restrictive, for two reasons: first, we wish to avoid some
 technicalities linked to the measurability of the control process; second, the assumptions on the costs must be consistent  with the assumptions that we shall make on the Hamiltonian, see Section~\ref{sec:assumptions} below. In particular, we shall impose that the Hamiltonian is globally Lipschitz continuous.  More general and difficult
cases, e.g., quadratic Hamiltonians, will be the subject of
a future work.

Consider a continuum of indistinguishable agents moving on the network $\Gamma$.
The  state of a representative agent at time $t$ is a time-continuous controlled stochastic
 process $X_t$
 as defined in Section~\ref{sec:class-stoch-proc}, where the control is the drift $a_t$, supposed to be in the form $a_t =a(X_t,t)$.
Let $m(\cdot, t)$ be the probability measure on $\Gamma$ that describes the distribution of states at time $t$.

For a representative agent, the optimal control problem  is of the form
\begin{equation}\label{eq:40}
v\left(x,t\right)=\inf_{a_s}\mathbb{E}_{xt}\left[\int_{t}^{T}\left(L\left(X_{s},a_{s}\right)+\ccV\left[m(\cdot,t)\right]\left(X_{s}\right)\right)ds+v_T\left(X_{T}\right)\right],
\end{equation}
where $\mathbb{E}_{xt}$ stands for the expectation conditioned by the event $X_t=x$.

We discuss the ingredients appearing in~\eqref{eq:40}:

\begin{itemize}
  \item We assume that the control is in a feeback form $a_t=a(X_t,t)$ where the function $a$, defined on $\Gamma\times [0,T]$,
      is sufficiently regular in the edges of the network. Then, almost surely if $X_t\in \Gamma_\alpha\backslash \cV$,
      \begin{displaymath}
        d  \pi_{\alpha}^{-1} (X_t)=  a_{\alpha}(\pi_{\alpha}^{-1} (X_t),t) dt + \sqrt{ 2\mu_{\alpha}} dW_t.
      \end{displaymath}
      An informal way to describe the behavior of the process at the vertices is  as follows: if $X_t$ hits $\nu_{i}\in\mathcal{V}$, then it enters $\Gamma_{\alpha}$,
      $\alpha\in\mathcal{A}_{i}$  with probability $p_{i\alpha}>0$ ($p_{i\alpha}$ was introduced in Section~\ref{sec:class-stoch-proc}). 
      We assume that there is an optimal feedback law $a^\star$.
  \item  We assume that for all $\alpha \in \cA$, $a_\alpha$  maps $[0,\ell_\alpha]\times [0,T] $ to a compact interval $A_\alpha = [\underline{a_\alpha}, \overline{a_\alpha}]$.
  \item
      The contribution of the control to the running cost involves the  Lagrangian $L$, i.e., a real valued function defined on 
      $\cup_{\alpha \in \mathcal{A}}  \left (\Gamma_\alpha \backslash \mathcal{V}\times A_\alpha \right) $. 
      If $x\in \Gamma_\alpha \backslash \mathcal{V}$ and $a\in A_ \alpha$, then $L(x,a)=L_\alpha(\pi_\alpha^{-1}(x),a)$, 
      where $L_\alpha$ is a continuous real valued function defined on $[0,\ell_\alpha]\times A_\alpha$.
      We assume that $L_\alpha(x,\cdot)$ is strictly convex on $A_\alpha$.
  \item
    The contribution of the  distribution of states to the running ccost  involves the coupling cost operator, which can either be nonlocal, i.e.,
    $\ccV:\mathcal{P}\left(\Gamma\right)\rightarrow   \mathcal{C}^2 (\Gamma)$
    (where $\mathcal{P}\left(\Gamma\right)$ is the set of  Borel probability
    measures on $\Gamma$),
    or local, i.e., $\ccV[m](x)= F(m(x))$
    for a continuous function $F: \R^+ \to \R$,
    assuming that $m$ is absolutely continuous with respect to the Lebesgue measure and identifying with its density.
    
  \item The last term is the terminal cost $v_T$, which depends only on the state variable for simplicity. 
    

\end{itemize}

Under suitable additional assumptions,  Ito calculus as in \cite{FS2000, FW1993}
 and the dynamic programming principle lead to the following
HJB equation on $\Gamma$,  more precisely the following boundary value problem
\begin{equation}
\begin{cases}
-\partial_{t}v-\mu_{\alpha}\partial^{2}v+H\left(x,\partial v\right)=\ccV[m(\cdot,t)](x), & \text{in }\left(\Gamma_{\alpha}\backslash\mathcal{V}\right)\times\left(0,T\right),\alpha\in \cA,\\
\ds
\sum_{\alpha\in\mathcal{A}_{i}}\gamma_{i\alpha}\mu_{\alpha}\partial_{\alpha}v\left(\nu_{i},t\right)=0, & \text{if }\left(\nu_{i},t\right)\in\mathcal{V}\times\left(0,T\right),\\
v|_{\Gamma_{\alpha}}\left(\nu_{i},t\right)=v|_{\Gamma_{\beta}}\left(\nu_{i},t\right) & \text{for all \ensuremath{\nu_{i}\in\mathcal{V},t\in\left(0,T\right)\alpha,\beta\in\mathcal{A}_{i}},}\\
v\left(x,T\right)=v_{T}\left(x\right) & \text{in }\Gamma.
\end{cases}\label{eq: introduction HJ}
\end{equation}
We refer to  \cite{LL2006-A,MR2271747,LL2007} for the interpretation of the value function $v$.
Let us comment  the different equations in~\eqref{eq: introduction HJ}:

\begin{enumerate}
\item The first equation is a HJB equation the Hamiltonian $H$ of which is a  real valued function defined
on $\left(\cup_{\alpha \in \mathcal{A}} \Gamma_\alpha \backslash \mathcal{V} \right) \times \R$ given by
\begin{eqnarray}\label{defHamilt}
H\left(x,p\right)=\sup_{a\in A_\alpha}\left\{ -a p-L_{\alpha}\left(\pi_\alpha^{-1}(x),a\right)\right\}
\quad \text{for $x\in \Gamma_\alpha \backslash \mathcal{V} $ and $p\in \R$}.
\end{eqnarray}
We assume that $L$ is such that
the Hamiltonians $H|_{\Gamma_\alpha\times \R}$ are Lipschitz continuous with
respect to $p$ and $C^1$.

\item  The second equation in \eqref{eq: introduction HJ} is  a Kirchhoff
transmission condition (or Neumann boundary condition if $\nu_i\in\partial\Gamma$);  it is the consequence of the assumption
on the behavior of $X_s$ at vertices. It involves the positive constants  $\gamma_{i\alpha}$ defined in~\eqref{eq:15}.
\item
 The third condition means in particular that $v$ is continuous at the vertices.
\item 
The fourth condition is a terminal condition for the backward in time HJB equation.
\end{enumerate}
If \eqref{eq: introduction HJ} has a smooth solution, then it provides a feedback law for the optimal control problem, i.e.,
 \[
 a^{\star}(x,t)=-\partial_{p}H\left(x,\partial v\left(x ,t \right)\right).
 \]
 At the MFG equilibrium, $m$ is the density of the invariant measure associated with the optimal feedback law, so,
 according to Section~\ref{sec:class-stoch-proc}, it satisfies~\eqref{eq:16}, 
 where $a$ is replaced by $a^\star=-\partial_{p}H\left(x,\partial v\left(x,t\right)\right)$. We end up with the following system:
\begin{equation}
{\displaystyle \begin{cases}
-\partial_{t}v-\mu_{\alpha}\partial^{2}v+H\left(x,\partial v\right)=\ccV\left[m(\cdot, t)\right](x), & \left(x,t\right)\in\left(\Gamma_{\alpha}\backslash\mathcal{V}\right)\times\left(0,T\right),\alpha\in \cA,\\
\partial_{t}m-\mu_{\alpha}\partial^{2}m-\partial\left(m\partial_{p}H\left(x,\partial v\right)\right)=0, & \left(x,t\right)\in\left(\Gamma_{\alpha}\backslash\mathcal{V}\right)\times\left(0,T\right),\alpha\in \cA,\\
{\displaystyle \sum_{\alpha\in\mathcal{A}_{i}}\gamma_{i\alpha}\mu_{\alpha}\partial_{\alpha}v\left(\nu_{i},t\right)=0,} & \left(\nu_{i},t\right)\in\mathcal{V}\times\left(0,T\right),\\
\ds
\sum_{\alpha\in\mathcal{A}_{i}}\mu_{\alpha}\partial_{\alpha}m\left(\nu_{i},t\right)+n_{i\alpha}\partial_{p}H^\alpha\left(\nu_{i},\partial v|_{\Gamma_{\alpha}}(\nu_i,t)\right)m|_{\Gamma_{\alpha}}\left(\nu_{i},t\right)=0, & \left(\nu_{i},t\right)\in\mathcal{V}\times\left(0,T\right),\\
v|_{\Gamma_{\alpha}}\left(\nu_{i},t\right)=v|_{\Gamma_{\beta}}\left(\nu_{i},t\right),\ \dfrac{m|_{\Gamma_{\alpha}}\left(\nu_{i},t\right)}{\gamma_{i\alpha}}=\dfrac{m|_{\Gamma_{\beta}}\left(\nu_{i},t\right)}{\gamma_{i\beta}}, & \alpha,\beta\in\mathcal{A}_{i},\left(\nu_{i},t\right)\in\mathcal{V}\times\left(0,T\right),\\
v\left(x,T\right)=v_{T}\left(x\right),\ m\left(x,0\right)=m_{0}\left(x\right) & x\in\Gamma,
\end{cases}}\label{eq: MFG system}
\end{equation}
where $H^{\alpha}:=H|_{\Gamma_\alpha\times \R}$.
At  a vertex $\nu_i$, $i\in I$, the transmission conditions for both  $v$ and $m$
 consist of $d_{\nu_{i}}= \sharp( \cA_i)$ linear relations, which is the appropriate number of relations to 
have a well posed problem. If $\nu_i\in \partial \Gamma$, there is of course only one Neumann like  condition for $v$ and for $m$.

\subsection{Assumptions and main results}

Before giving the precise definition of solutions of the MFG system~\eqref{eq: MFG system}
and stating our result, we need to introduce some suitable functions spaces.

\subsubsection{Function spaces related to the Kirchhoff conditions}
\label{sec:funct-spac-relat}
The following function spaces will be the key ingredients in order to build  weak
solutions of~\eqref{eq: MFG system}.
\begin{defn}
\label{def: functional spaces}
We define two Sobolev spaces:
$
V: = H^1(\Gamma), 
$
and 
\begin{equation}
W:=\left\{ w:\Gamma\rightarrow\mathbb{R}:\;w\in H_{b}^{1}\left(\Gamma\right)  \text{  and }\dfrac{w|_{\Gamma_{\alpha}}\left(\nu_{i}\right)}{\gamma_{i\alpha}}=\dfrac{w|_{\Gamma_{\beta}}\left(\nu_{i}\right)}{\gamma_{i\beta}}\text{ for all }
i\in I, \; \alpha,\beta\in\mathcal{A}_{i}\right\}, \label{eq: space W}
\end{equation}
which is a subspace of $H^{1}_{b}(\Gamma)$.
\end{defn}

\begin{defn}
\label{def: test function}
Let the function $\varphi\in W$ be defined as follows:
\begin{equation}
\begin{cases}
\varphi_{\alpha}\text{ is affine on }\left(0,\ell_{\alpha}\right),\\
\varphi|_{\Gamma_{\alpha}}\left(\nu_{i}\right)=\gamma_{i\alpha},\text{ if }\alpha\in\mathcal{A}_{i},\\
\varphi\text{ is constant on the edges \ensuremath{\Gamma_{\alpha}} which touch the boundary of \ensuremath{\Gamma}}.
\end{cases}\label{eq: test function 1}
\end{equation}
Note that $\varphi$ is positive and bounded. 
We  set $\overline{\varphi}=\max_{\Gamma}\varphi$, $\underline{\varphi}=\min_{\Gamma}\varphi$.
\end{defn}

\begin{rem}
\label{rem: test function}One can see that $v\in V\longmapsto v\varphi$
is an isomorphism from $V$ onto $W$ and $w\in W\longmapsto w\varphi^{-1}$
is the inverse isomorphism.
\end{rem}

\begin{defn}\label{sec:funct-spac-relat-1}
Let the function space $\mathcal{W}\subset W$  be defined as follows:
\begin{equation}
\mathcal{W}:=\left\{ m:\Gamma\to \R: m_{\alpha}\in C^{1}\left(\left[0,\ell_{\alpha}\right]\right)\text{ and }\dfrac{m|_{\Gamma_{\alpha}}\left(\nu_{i}\right)}{\gamma_{i\alpha}}=\dfrac{m|_{\Gamma_{\beta}}\left(\nu_{i}\right)}{\gamma_{i\beta}}\text{ for all \ensuremath{i\in I,\alpha,\beta\in\mathcal{A}_{i}}}\right\} .\label{eq: regular space}
\end{equation}
\end{defn}

\subsubsection{Running assumptions (H)}

\label{sec:assumptions}

\begin{itemize}
  
\item[] (Diffusion constants) $(\mu_\alpha)_{\alpha\in \mathcal{A}}$ is a family of positive numbers.

\item[] (Jump coefficients)  $( \gamma_{i\alpha})_{\alpha\in\mathcal{A}_{i}} $ is a family of positive numbers such that 
  $\ds\sum_{\alpha\in \cA_i} \gamma_{i\alpha}\mu_\alpha=1$.

\item[] (Hamiltonian)   The Hamiltonian $H$ is defined by the collection $H^{\alpha}:=H|_{\Gamma_\alpha\times \R}$, $\alpha\in \cA$: we assume that 
\begin{align}
 & H^{\alpha}\in C^{1}\left(\Gamma_\alpha \times\mathbb{R}\right),\label{eq: Hamiltonian}\\
 & H^{\alpha}\left(x,\cdot\right)\text{is convex in }p, &\text{ for any }x\in \Gamma_\alpha,\label{eq: Hamiltonian is convex}\\
 & H^{\alpha}\left(x,p\right)\le C_{0}(\left|p\right|+1),&\text{ for any }\left(x,p\right)\in \Gamma_\alpha\times\mathbb{R},\label{eq:17} \\
 & \left|\partial_{p}H^{\alpha}\left(x,p\right)\right|\le C_{0}, &\text{ for any }\left(x,p\right)\in \Gamma_\alpha\times\mathbb{R},\label{eq:18}\\
 & \left|\partial_{x}H^{\alpha}\left(x,p\right)\right|\le C_{0} (|p|+1), & \text{ for any }\left(x,p\right)\in \Gamma_\alpha\times\mathbb{R},\label{eq:19}
\end{align}
for a constant $C_0$ independent of $\alpha$.
\item[](Coupling operator)
We assume that $\ccV$ is a continuous  map from $L^2(\Gamma)$ to $L^2(\Gamma)$, such that for all $m\in L^2(\Gamma)$,
\begin{equation}\label{eq:20}
\|\ccV[m]\|_{L^2(\Gamma)}\le C (\|m\|_{L^2(\Gamma)}+1).
\end{equation}
Note that such an assumption is satisfied by local operators of the form $\ccV[m](x)=F(m(x))$ where $F$ is a Lipschitz-continuous function.

\item[] (Initial and terminal data)  $m_0\in L^2(\Gamma)\cap \mathcal{M}$ and $v_T\in H^1(\Gamma)$.
\end{itemize}
The above set of assumptions, referred to as (H),  will be the running assumptions hereafter.
We will use the following notation: $\underline{\mu}:= {\rm min}_{\alpha \in \mathcal{A}}\,\mu_\alpha >0$
and $\overline{\mu}:= {\rm max}_{\alpha \in \mathcal{A}}\, \mu_\alpha$.

\subsubsection{Strictly increasing coupling}
\label{sec:strict-increas}
              
\begin{def}
\label{sec:running-assumptions}
  We will also say that the coupling $\ccV$ is strictly increasing if, for any $m_1,m_2\in \mathcal{M}\cap L^2(\Gamma)$,
\begin{eqnarray*}
\int_\Gamma  (m_1-m_2)(\ccV[m_1]-\ccV[m_2])dx \geq 0
\end{eqnarray*}
and equality implies $m_1=m_2$. 
\end{def}

\subsubsection{ Stronger assumptions on the coupling operator}\label{sec:strong-assumpt-coupl}
We will sometimes need to strengthen the assumptions on the coupling operator, namely 
that $\ccV$ has the following  smoothing properties:

 $\ccV$ maps the topological dual of $W$ to $H^1_b(\Gamma)$; more precisely, $\ccV$ defines 
a   Lipschitz map from 
 $W'$ to $H^1_b(\Gamma)$.

Note that such an assumption is not satisfied by local operators.


\subsubsection{Definition of solutions and main result}
\label{sec:main-result}

\begin{defn} (solutions of the MFG system)
  \label{sec:main-result-1}
A weak solution of the Mean Field Games system \eqref{eq: MFG system}
is a pair $\left(v,m\right)$ such that
\begin{eqnarray*}
&&v\in L^{2}\left(0,T;H^{2}\left(\Gamma\right)\right)  \cap C([0,T]; V),
\ \partial_{t}v\in L^{2}\left(0,T;L^{2}\left(\Gamma\right)\right),\\
&&m\in L^{2}\left(0,T;W\right)\cap  C((0,T]; L^2(\Gamma)\cap \mathcal{M} ),\ \partial_{t}m\in L^{2}\left(0,T;V'\right), 
\end{eqnarray*}
$v$ satisfies
\begin{eqnarray*}
  \left\{
  \begin{array}{l}
\ds -\sum_{\alpha\in\mathcal{A}}\int_{\Gamma_{\alpha}}\left[\partial_{t}v\left(x,t\right)\mathsf{w}\left(x\right)+\mu_{\alpha}\partial v\left(x,t\right)\partial\mathsf{w}\left(x\right)+H\left(x,\partial v\left(x,t\right)\right)\mathsf{w}\left(x\right)\right]dx\\
\hspace*{4cm}\ds  =  \int_{\Gamma} \ccV[m(\cdot,t)](x) \mathsf{w}\left(x\right)dx,
  \quad\text{for all \ensuremath{\mathsf{w}\in W},  a.e. $t\in (0,T)$,}\\
v(x,T)=v_T(x)  \quad\text{for a.e. $x\in\Gamma$},
\end{array}
  \right.
\end{eqnarray*}
and $m$ satisfies
\begin{eqnarray*}
  \left\{
  \begin{array}{l}
\ds \sum_{\alpha\in\mathcal{A}}\int_{\Gamma_{\alpha}}\left[\partial_{t}m\left(x,t\right)\mathsf{v}\left(x\right)dx+\mu_{\alpha}\partial m\left(x,t\right)\partial\mathsf{v}\left(x\right)+\partial_{p}H\left(x,\partial v\left(x,t\right)\right)m\left(x,t\right)\partial\mathsf{v}\left(x\right)\right]dx\\
\hspace*{8cm}\ds  =  0,\quad\text{for all \ensuremath{\mathsf{v}\in V}, a.e. $t\in (0,T)$,} \\
m(x,0)=m_0(x)\quad\text{for a.e. $x\in\Gamma$},
\end{array}
  \right.
\end{eqnarray*}
where $V$ and $W$ are introduced in Definition~\ref{def: functional spaces}.
\end{defn}


We are ready to state the main result:
\begin{thm}
  \label{thm: MFG system} Under assumptions (H), 
  \begin{itemize}
  \item[(i)] (Existence)
    There exists a weak solution $\left(v,m\right)$
    of~\eqref{eq: MFG system}. 
  \item[(ii)] (Uniqueness) If $\ccV$ is strictly increasing (see~\ref{sec:strict-increas}), then the solution is unique.
  \item[(iii)]  (Regularity) If $\ccV$ satisfies furthermore the stronger assumptions made in Section \ref{sec:strong-assumpt-coupl}  and 
    if $v_T\in C^{2+\eta}(\Gamma)\cap D$ for some $\eta\in (0,1)$ ($D$ is  given in (\ref{Kirchhoff condition})),
    then $v\in C^{2,1}(\Gamma\times [0,T])$. \\
Moreover, if  for all $\alpha \in \cA$, $\partial_p  H^\alpha (x, p)$ is a Lipschitz
 function defined in $\Gamma_\alpha \times \R$, and if $m_0\in W$,
 then  $m\in C([0,T];W)\cap W^{1,2}(0,T; L^2(\Gamma))\cap L^2 (0,T; H^2_b(\Gamma))$.
 
  \end{itemize}
\end{thm}

\section{\label{subsec: The-linear-parabolic}Preliminary: a  modified heat equation on the
 network with general Kirchhoff conditions
}

This section contains results on the solvability of some linear boundary
value problems with terminal
condition, that will be useful in what follows. Consider
\begin{equation}
\begin{cases}
-\partial_{t}v-\mu_{\alpha}\partial^{2}v=h, & \text{in }\left(\Gamma_{\alpha}\backslash\mathcal{V}\right)\times\left(0,T\right),\alpha\in\mathcal{A},\\
v|_{\Gamma_{\alpha}}\left(\nu_{i},t\right)=v|_{\Gamma_{\beta}}\left(\nu_{i},t\right), &  t\in\left(0,T\right)\alpha,\beta\in\mathcal{A}_{i},\nu_{i}\in\mathcal{V},\\
\ds \sum_{\alpha\in\mathcal{A}_{i}}\gamma_{i\alpha}\mu_{\alpha}\partial_{\alpha}v\left(\nu_{i},t\right)=0, & t\in\left(0,T\right),\nu_{i}\in\mathcal{V},\\
v\left(x,T\right)=v_{T}\left(x\right), & x\in\Gamma,
\end{cases}\label{eq: linear 1}
\end{equation}
where $h\in L^{2}\left(0,T;W'\right)$ and $v_{T}\in L^2(\Gamma)$.

\begin{defn}
If $v_T\in L^2(\Gamma)$ and  $h\in L^{2}\left(0,T; W'\right)$,  
a weak solution of \eqref{eq: linear 1} is a function $v\in L^{2}\left(0,T;V\right)\cap C([0,T]; L^2 (\Gamma))$ such that
$\partial_{t}v\in L^{2}\left(0,T;W'\right)$ and 
\begin{equation}
\begin{cases}
\ds -\left\langle \partial_{t}v\left(t\right), w\right \rangle_{W', W} +\mathscr{B}\left(v\left(\cdot,t\right),w\right)=
\left\langle h(t), w \right \rangle_{W', W} \quad\text{for all }w\in W\text{ and a.e. }t\in (0,T),\\
v\left(x,T\right)=v_{T}(x),
\end{cases}\label{eq: linear 1 weak form}
\end{equation}
where $\mathscr{B}:V\times W\rightarrow\mathbb{R}$ is the bilinear
form defined as follows:
\[
\mathscr{B}\left(v,w\right):=\int_{\Gamma}\mu\partial v\partial wdx=\sum_{\alpha\in\mathcal{A}}\int_{\Gamma_{\alpha}}\mu_{\alpha}\partial v\partial wdx.
\]
\end{defn}

We use the Galerkin's method (see~\cite{Evans2010}), i.e., we construct solutions of some
finite-dimensional approximations to \eqref{eq: linear 1}.

Recall that $\varphi$ has been defined in Definition \ref{def: test function}.
We notice first that the symmetric bilinear form $\widehat \cB( u,v ):= \int_\Gamma \mu \varphi \partial u \partial v$ is 
such that $(u,v)\mapsto  (u,v)_{L^2 (\Gamma)}+ \widehat \cB( u,v )$ is an inner product in $V$ equivalent to the standard inner product
in $V$, namely $(u,v)_{V}= (u,v)_{L^2 (\Gamma)} + \int_\Gamma  \partial u \partial v$. Therefore, by standard Fredholm's theory, there exist 
\begin{itemize}
\item a non decreasing sequence  of nonnegative real numbers $(\lambda_k)_{k=1}^{\infty}$, that tends to $+\infty$ as $k\to \infty$
\item A Hilbert basis $\left( \mathsf{v}_{k}\right) _{k=1} ^{\infty}$  of $L^2(\Gamma)$ , which is also a 
 a total sequence   of $V$ (and orthogonal if $V$ is endowed with the scalar product  $  (u,v)_{L^2 (\Gamma)}+ \widehat \cB( u,v )$),
\end{itemize}
such that 
\begin{equation}
  \label{eq:21}
\widehat \cB( \mathsf{v}_{k} ,v )= \lambda_k  (\mathsf{v}_{k} ,v)_{L^2 (\Gamma)} \quad \hbox{for all } v\in V.
\end{equation}
Note that 
\[
 \int_{\Gamma}\mu\partial\mathsf{v}_{k}\partial\mathsf{v}_{\ell}\varphi dx=\begin{cases}
 \lambda_{k} & \text{if \ensuremath{k=\ell}},\\
 0 & \text{if \ensuremath{k\ne\ell.}}
\end{cases}
 \]
Note also that $\mathsf{v}_{k}$ is a weak solution of
 \begin{equation}
 \begin{cases}
 -\mu_{\alpha}\partial\left(\varphi\partial\mathsf{v}_{k}\right)=\lambda_{k}\mathsf{v}_{k}, & \text{in }\Gamma_{\alpha}\backslash\mathcal{V},\alpha\in\mathcal{A},\\
 \mathsf{v}_{k}|{}_{\Gamma_{\alpha}}\left(\nu_{i}\right)=\mathsf{v}_{k}|{}_{\Gamma_{\beta}}\left(\nu_{i}\right), & \alpha,\beta\in\mathcal{A}_{i},\\
 \ds\sum_{\alpha\in\mathcal{A}}\gamma_{i\alpha}\mu_{\alpha} \partial_{\alpha}\mathsf{v}_{k}\left(\nu_{i}\right)=0, & \nu_{i}\in\mathcal{V},
 \end{cases}\label{eq: eigenvalue problem}
 \end{equation}
which implies that $\mathsf{v}_{k} \in C^2 (\Gamma)$.

Finally, by Remark~\ref{rem: test function}, the sequence $(\varphi \mathsf{v}_{k} )_{k=1}^{\infty}$ is a total family in $W$ (but is not orthogonal if $W$ is endowed with the standard inner product).

\begin{lem}
\label{lem: Galerkin's method}For any positive integer $n$, there
exist $n$ absolutely continuous functions $y_{k}^{n}: [0,T]\to \R$ ,  $k=1,\dots, n$,  
and a function $v_{n}:\left[0,T\right]\to L^2(\Gamma)$ of the form
\begin{equation}
v_{n}\left(x,t\right)=\sum_{k=1}^{n}y_{k}^{n}\left(t\right)\mathsf{v}_{k}(x),\label{eq: formula of v_n}
\end{equation}
such that
\begin{equation}
y_{k}^{n}\left(T\right)=\int_{\Gamma}v_{T}\mathsf{v}_{k}dx,\quad\text{for }k=1,\dots,n,\label{eq: aproximate terminal condtion v_n}
\end{equation}
and 
\begin{equation}
- \frac d {dt } (v_{n},\mathsf{v}_{k}\varphi)_{L^2 ( \Gamma)} + \mathscr{B}\left(v_{n},\mathsf{v}_{k}\varphi\right)=
\left\langle h (t),\mathsf{v}_{k}\varphi \right \rangle,\quad  \hbox {for a.a. }t\in\left(0,T\right), \hbox { for  all } k=1,\dots, n.\label{eq: approximate linear equation}
\end{equation}
\end{lem}

\begin{proof}[Proof of Lemma~\ref{lem: Galerkin's method}]
For $n\ge1$, we consider the symmetric $n$ by $n$ matrix $M_{n}$
defined by 
\[
\left(M_{n}\right)_{k\ell}=\int_{\Gamma}\mathsf{v}_{k}\mathsf{v}_{\ell}\varphi dx.
\]
Since $\varphi$ is positive and $ (\mathsf{v}_{k}) _{k=1}^{\infty}$
is a Hilbert basis of $L^{2}\left(\Gamma\right)$, we can check
that $M_n$ is a positive definite matrix and there exist two constants
$c,C$ independent of $n$ such that
\begin{equation}
  c\left|\xi\right|^{2}\le\sum_{k,\ell=1}^{n}\left(M_{n}\right)_{k\ell}\xi_{k}\xi_{\ell}\le C\left|\xi\right|^{2},
  \quad\text{for all }\xi\in\R^n.
\end{equation}
Looking for $v_n$ of the form~\eqref{eq: formula of v_n},
and setting $Y=\left(y_{1}^{n},\ldots,y_{n}^{n}\right)^{T}$,
$\dot{Y}=\left( \frac d {dt} y_{1}^{n},\ldots,\frac d {dt}y_{n}^{n}\right)^T$, \eqref{eq: approximate linear equation}~implies that we have to solve the following
 system of differential equations
\[
-M_n\dot{Y}+B Y=F_{n}, \quad
Y(T)= \left(\int_{\Gamma}{v}_{T}\mathsf{v}_{1},\ldots,\int_{\Gamma}{v}_{T}\mathsf{v}_{n}\right)^T,
\]
where $B_{k\ell}=\mathscr{B}\left(\mathsf{v}_{\ell},\mathsf{v}_{k}\varphi\right)$
and $F_{n}(t)=\left( \langle h(t),\mathsf{v}_{1}\varphi\rangle ,\ldots, \langle h(t),\mathsf{v}_{n}\varphi \rangle \right)^{T}$.
Since the matrix $M_n$ is invertible, the ODE system has a unique absolutely
continuous solution. The lemma is proved.
\end{proof}

We propose to send $n$ to $+\infty$ and  show that a subsequence of
$\left\{ v_{n}\right\} $ converges to a solution of~\eqref{eq: linear 1}.
Hence, we need some uniform estimates for $\left\{ v_{n}\right\} $.

\begin{lem}
\label{lem: Energy estimate}There exists a constant $C$ depending
only on $\Gamma$, $(\mu_\alpha)_{\alpha\in \cA}$, $T$ and $\varphi$ such that 
\[
\left\Vert v_{n}\right\Vert _{L^{\infty}\left(0,T;L^{2}\left(\Gamma\right)\right)}+\left\Vert v_{n}\right\Vert _{L^{2}\left(0,T;V\right)}+\left\Vert \partial_{t}v_{n}\right\Vert _{L^{2}\left(0,T;W'\right)}\le C\left(\left\Vert h\right\Vert
_{L^{2}\left(0,T;W'\right)}
+\left\Vert v_{T}\right\Vert _{L^{2}\left(\Gamma\right)}\right).
\]
\end{lem}

\begin{proof}[Proof of Lemma~\ref{lem: Energy estimate}]
Multiplying \eqref{eq: approximate linear equation} by $y_{k}^{n}\left(t\right)e^{\lambda t}$
for a positive constant $\lambda$ to be chosen later, summing for $k=1,\dots,n$
and using the formula~\eqref{eq: formula of v_n} for $v_{n}$, we get
\[
-\int_{\Gamma}\partial_{t}v_{n}v_{n}e^{\lambda t}\varphi dx+\int_{\Gamma}\mu\partial v_{n}\partial\left(v_{n}e^{\lambda t}\varphi\right)dx=e^{\lambda t} \langle h(t) , v_{n}\varphi \rangle_{W',W},
\]
and 
\begin{align*}
-\int_{\Gamma}\left[\partial_{t}\left(\dfrac{v_{n}^{2}}{2}e^{\lambda t}\right)-\dfrac{\lambda}{2}v_{n}^{2}e^{\lambda t}\right]\varphi dx+\int_{\Gamma}\mu\left(\partial v_{n}\right)^{2}e^{\lambda t}\varphi dx+\int_{\Gamma}\mu\partial v_{n}v_{n}e^{\lambda t}\partial\varphi dx=e^{\lambda t} \langle h(t) , v_{n}\varphi \rangle .
\end{align*}
Integrating both sides from $s$ to $T$, we obtain
\begin{eqnarray}
 && \int_{\Gamma}\left(\dfrac{v_{n}^{2}\left(x,s\right)}{2}e^{\lambda s}
-\dfrac{v_{n}^{2}\left(x,T\right)}{2}e^{\lambda T}\right)\varphi dx
+\dfrac{\lambda}{2}\int_{s}^{T}\int_{\Gamma}v_{n}^{2}e^{\lambda t}\varphi dxdt \nonumber \\
 && +\int_{s}^{T}\int_{\Gamma}\mu(\partial v_{n})^{2}e^{\lambda t}\varphi dxdt+
\int_{s}^{T}\int_{\Gamma}\mu\partial v_{n}v_{n}e^{\lambda t}\partial\varphi dxdt\nonumber \\
&= & \int_{s}^{T} e^{\lambda t} \langle h(t) , v_{n}(t)\varphi \rangle dt\nonumber \\
&\le & C \int_{s}^{T} e^{\lambda t} \|h(t)\|_{W'}  \|v_{n}(t)\|_{V} dt\nonumber \\
& \le & 
\frac 1 2 \int_{s}^{T}\int_{\Gamma}\mu \left( (\partial v_{n})^{2} + v_{n}^{2} \right) e^{\lambda t}\varphi dxdt
+ \frac {C^2}{2 \underline{\mu}}   \int_{s}^{T}e^{\lambda t} \|h(t)\|^2_{W'}dt,  \nonumber 
\end{eqnarray}
 where $C$ is positive constant depending on $\varphi$, because of Remark \ref{rem: test function}. Therefore,
\begin{eqnarray*}
 && e^{\lambda s} \int_{\Gamma} \dfrac{v_{n}^{2}\left(x,s\right)}{2} \varphi dx
+ \frac {1} 4 \int_{s}^{T}\int_{\Gamma}\mu(\partial v_{n})^{2}\varphi e^{\lambda t} dxdt
+\left( \frac {\lambda} 2 - \frac { \overline {\mu}} 2 -  \overline {\mu} \frac {\|\partial \varphi \|^2_{L^\infty(\Gamma)}}{\underline \varphi^2}   \right) 
\int_{s}^{T}\int_{\Gamma}v_{n}^{2}e^{\lambda t}\varphi dxdt\nonumber \\
 &\le & 
e^{\lambda T}\int_{\Gamma} \dfrac{v_{n}^{2}\left(x,T\right)}{2}\varphi dx
+ \frac {C^2}{2 \underline{\mu}} e^{\lambda T}  \int_{s}^{T} \|h(t)\|^2_{W'}dt.
\end{eqnarray*}
Choosing $\lambda \ge 1/2+ \overline{\mu} + 2\overline{\mu}||\partial \varphi||^2_{L^\infty(\Gamma)}/\underline{\varphi}^2$
and noticing that $\int_{\Gamma}v_{n}^{2}\left(x,T\right)\varphi dx$ is bounded
by $\overline{\varphi}\int_{\Gamma}v_{T}^{2}dx$ from~\eqref{eq: aproximate terminal condtion v_n}, it follows  that
\begin{eqnarray}
  && \int_{\Gamma}v_{n}^{2}(x,s)\varphi dx + \int_{s}^{T}\int_{\Gamma}v_{n}^{2}\varphi dxdt
  +  \int_{s}^{T}\int_{\Gamma}\mu (\partial v_{n})^{2}\varphi dxdt\nonumber\\
  &\leq&
2
  e^{\lambda T}\left( 
  \frac {C^2} {\underline \mu} \Vert h \Vert _{L^{2}(0,T;W')}^{2}  +\overline{\varphi} \int_{\Gamma}v_{T}^{2}dx \right) . 
\label{ineq:useful-formula}
\end{eqnarray}

\emph{Estimate of $v_{n}$ in $L^{\infty}\left(0,T;L^{2}\left(\Gamma\right)\right)$ and $L^{2}\left(0,T;V\right)$}.
From~\eqref{ineq:useful-formula}, it is straightforward to see that
\begin{eqnarray}\label{ineq:042}
  && \left\Vert v_{n}\right\Vert _{L^{\infty}\left(0,T;L^{2}\left(\Gamma\right)\right)}
  +\left\Vert v_{n}\right\Vert _{L^{2}\left(0,T;V\right)}
\leq C\left(\left\Vert h\right\Vert _{L^{2}(0,T;W')}+\left\Vert v_{T}\right\Vert _{L^{2}\left(\Gamma\right)}\right)
\end{eqnarray}
for some constant $C$ depending only on $(\mu_\alpha)_{\alpha\in \cA}$, $\varphi$ and $T$.

\emph{Estimate $\partial_{t}v_{n}$ in $L^{2}\left(0,T;W'\right)$}.
Consider the closed subspace $G_1$ of $W$ defined by
$G_1=\left\{ w\in W:\int_{\Gamma}\mathsf{v}_{k}wdx=0\text{ for all }k\le n\right\}$. It has a finite co-dimension equal to $n$. Consider also the $n$-dimensional subspace $G_2=\text{span}\left\{ \mathsf{v}_{1}\varphi,\ldots,\mathsf{v}_{n}\varphi\right\} $ of $W$.  The invertibility of the matrix $M_n$ introduced in the proof of Lemma \ref{lem: Galerkin's method} implies that $G_1\cap G_2=\{0\}$. This implies that $W= G_1\oplus G_2$.
For $w\in W$, we can write $w$ of the form $w=w_{n}+\hat{w}_{n}$,
where $w_{n}\in  G_2$ and  $\hat{w}_{n} \in G_1$.
Hence, for a.e. $t\in\left[0,T\right]$, from \eqref{eq: formula of v_n}
and \eqref{eq: approximate linear equation}, one gets
\begin{equation}
\left\langle \partial_ t v_{n}(t),w\right\rangle _{W',W} =
\frac {d} {dt} \left(\int_{\Gamma}v_{n}wdx\right)= \frac {d} {dt} \left(\int_{\Gamma} v_{n}w_{n}dx\right) 
=-\langle h(t),w_{n}\rangle_{W',W} +\int_{\Gamma}\mu\partial v_{n}\partial w_{n}dx.\label{eq: estimate v_t in W'}
\end{equation}
Since there exists a constant $C$ independent of $n$ such that $\left\Vert w_{n}\right\Vert _{W}\le C\left\Vert w\right\Vert _{W}$,
it follows that 
\[
\left\Vert \partial_{t}v_{n}\left(t\right)\right\Vert _{W'}\le C\left(\left\Vert h\left(t\right)\right\Vert _{
W'}+\overline{\mu}\left\Vert v_{n}\left(t\right)\right\Vert _{V}\right),
\]
for almost every $t$, and therefore, from~\eqref{ineq:042}, we obtain
\begin{displaymath}
\left\Vert \partial_{t}v_{n}\left(t\right)\right\Vert_{L^{2}\left(0,T;W'\right)}^{2}  
 \le C\left(\left\Vert h\right\Vert _{L^{2}\left(0,T;W'\right)}^{2}+\left\Vert v_{T}\right\Vert _{L^{2}\left(\Gamma\right)}^{2}\right),  
\end{displaymath}
for a constant $C$ independent of  $n$.
\end{proof}

\begin{thm}
\label{thm: existence and uniqueness linear equation}There exists
a unique solution $v$ of \eqref{eq: linear 1}, which satisfies
\begin{equation}
\left\Vert v\right\Vert _{L^{\infty}\left(0,T;L^2(\Gamma)\right)}
+\left\Vert v\right\Vert _{L^{2}\left(0,T;V\right)}+\left\Vert \partial_{t}v\right\Vert _{L^{2}\left(0,T;W'\right)}\le C\left(\left\Vert h\right\Vert _{L^{2}\left(0,T; W'\right)}+\left\Vert v_{T}\right\Vert _{L^{2}\left(\Gamma\right)}\right),\label{eq: energy estimate for v}
\end{equation}
where $C$ is a constant that depends only on $\Gamma$, $(\mu_\alpha)_{\alpha\in \cA}$, $T$ and $\varphi$. 
\end{thm}

\begin{proof}[Proof of  Theorem~\ref{thm: existence and uniqueness linear equation}]
From Lemma~\ref{lem: Energy estimate}, the sequence $\left( v_{n}\right) _{n\in\mathbb{N}}$
is bounded in $L^{2}\left(0,T;V\right)$ and the sequence $\left( \partial_{t}v_{n}\right) _{n\in\mathbb{N}}$
is bounded in $L^{2}(0,T;W')$. Hence, up to the extraction of a subsequence,
there exists a function $v$ such that $v\in L^{2}\left(0,T;V\right),\;\partial_{t}v\in L^{2}\left(0,T;W'\right)$
and
\begin{equation}
\left\{
\begin{array}{ll}
v_{n} \rightharpoonup v & \text{ weakly in }L^{2}\left(0,T;V\right),\\
\partial_{t}v_{n} \rightharpoonup\partial_{t}v & \text{ weakly in }L^{2}\left(0,T;W'\right).
\end{array}\label{eq: weakly convergence for v_n}
\right.
\end{equation}
Fix an integer $N$ and choose a function $\overline{v}\in C^{1}\left(\left[0,T\right];V\right)$
having the form
\begin{equation}
\overline{v}\left(t\right)=\sum_{k=1}^{N}d_{k}\left(t\right)\mathsf{v}_{k},\label{eq: formula v_bar}
\end{equation}
where $d_{1},\ldots,d_{N}$ are given  real valued $C^1$ functions defined in $[0,T]$. For all $n\ge N$,
multiplying \eqref{eq: approximate linear equation} by $d_{k}\left(t\right)$,
summing for $k=1,\dots,n$ and integrating over $\left(0,T\right)$ leads to
\begin{equation}
-\int_{0}^{T}\int_{\Gamma}\partial_{t}v_{n}\overline{v}\varphi dxdt+\int_{0}^{T}\int_{\Gamma}\mu\partial v_{n}\partial\left(\overline{v}\varphi\right)dxdt=\int_{0}^{T} \langle h,\overline{v}\varphi \rangle dt.\label{eq: v_n and v_bar}
\end{equation}
Letting $n\rightarrow+\infty$,
we obtain from   \eqref{eq: weakly convergence for v_n} that
\begin{equation}
  -\int_{0}^{T}\langle \partial_{t}v,\overline{v}\varphi \rangle dt+\int_{0}^{T}\int_{\Gamma}\mu\partial v\partial\left(\overline{v}\varphi\right)dxdt=\int_{0}^{T}\langle h,\overline{v}\varphi \rangle dt.
\label{eq: v and v_bar}
\end{equation}
Since the functions of the form \eqref{eq: formula v_bar}
are dense in $L^{2}\left(0,T;V\right)$,
(\ref{eq: v and v_bar}) holds for all test function $\overline{v}\in L^{2}\left(0,T;V\right)$.
Recalling the isomorphism $\overline{v}\in V\mapsto \overline{v}\varphi\in W$ (see Remark~\ref{rem: test function}),
we obtain that, for all $w\in W$ and $\psi\in C_{c}^{1}\left(0,T\right)$,
\[
-\int_{0}^{T}\langle \partial_{t}v, w\rangle \psi dt+\int_{0}^{T}\int_{\Gamma}\mu\partial v\partial w\psi dxdt=\int_{0}^{T}
\langle h,w \rangle \psi dt.
\]
This implies that, for a.e. $t\in\left(0,T\right)$,
\[
-\langle \partial_{t}v, w \rangle +\mathscr{B}\left(v,w\right)=\langle h,w \rangle\quad\text{for all }w\in W.
\]

Using \cite[Theorem 3.1]{lm68} (or  the same argument as in \cite[pages 287-288]{Evans2010}), 
we see that  $v\in C([0,T]; L^2_\varphi(\Gamma))$, where
$L^2_\varphi(\Gamma)= \{w:\Gamma\to \R\;:\;\int_\Gamma w^2 \varphi dx <+\infty\}$, and since $\varphi$ is bounded from below and from above by positive numbers,
$L^2_\varphi(\Gamma)=L^2(\Gamma)$ with equivalent norms. Moreover, 
\begin{displaymath}
  \max_{0\le t\le T} \|v(\cdot,t)\|_{L^2(\Gamma)}\le C \left(\|\partial_t v\|_{L^2(0,T;W')}+\|v\|_{L^2(0,T;V)}\right).
\end{displaymath}
 We are now going to prove $v\left(T\right)=v_{T}$.
 For all $\overline{v}\in C^{1}\left(\left[0,T\right];V\right)$ of the form 
(\ref{eq: formula v_bar}) and such that $\overline{v}\left(0\right)=0$, we deduce from \eqref{eq: v_n and v_bar} and \eqref{eq: v and v_bar} that
\begin{align*}
 & -\int_{0}^{T}\int_{\Gamma}\partial_{t}\overline{v}v_{n}\varphi dxdt-\int_{\Gamma}\overline{v}\left(T\right)v_{n}\left(T\right)\varphi dx+\int_{0}^{T}\int_{\Gamma}\mu\partial v_{n}\partial\left(\overline{v}\varphi\right)dxdt\\
= & -\int_{0}^{T}\int_{\Gamma}\partial_{t}\overline{v}v\varphi dxdt-\int_{\Gamma}\overline{v}\left(T\right)v\left(T\right)\varphi dx+\int_{0}^{T}\int_{\Gamma}\mu\partial v\partial\left(\overline{v}\varphi\right)dxdt.
\end{align*}
We know that $v_{n}\left(T\right)\rightarrow v_{T}$ in $L^{2}\left(\Gamma\right)$. Then, 
using \eqref{eq: weakly convergence for v_n}, we obtain 
\[
\int_{\Gamma}\overline{v}\left(T\right)v_{T}\varphi dx=\int_{\Gamma}\overline{v}\left(T\right)v\left(T\right)\varphi dx.
\]
Since the functions of the form $\sum_{k=1}^{N}d_{k}\left(T\right)\mathsf{v}_{k}$ are dense in $L^2(\Gamma)$,
 we conclude that $v\left(T\right)=v_{T}$. 

In order to prove the energy estimate \eqref{eq: energy estimate for v},
we use $ve^{\lambda t}\varphi$
as a test function in (\ref{eq: linear 1 weak form}) and apply  similar arguments as in the proof of Lemma \ref{lem: Energy estimate}
for $\lambda$ large enough, we get~\eqref{eq: energy estimate for v}.

Finally, if $h= 0$ and $v_T=0$, by the energy estimate for $v$ in
\eqref{eq: energy estimate for v}, we deduce that $v= 0$. Uniqueness
is proved.
\end{proof}

\begin{thm}
\label{thm: regularity for v} If $v_{T}\in V$ and $h\in L^{2}\left( \Gamma\times (0,T)\right)$,
then the unique solution $v$ of \eqref{eq: linear 1} satisfies $v\in L^{2}\left(0,T;H^{2}\left(\Gamma\right)\right)\cap
C([0,T];V)$
and $\partial_{t}v\in L^{2}\left(0,T;L^{2}\left(\Gamma\right)\right)$. Moreover,
\begin{eqnarray}\label{meilleure-esti-v}
\left\Vert v\right\Vert _{L^{\infty}\left(0,T;V\right)}+\left\Vert v\right\Vert _{L^{2}\left(0,T;H^{2}\left(\Gamma\right)\right)}+\left\Vert \partial_{t}v\right\Vert _{L^{2}\left(0,T;L^{2}\left(\Gamma\right)\right)}\le C\left(\left\Vert h\right\Vert _{L^{2}\left(0,T;L^{2}\left(\Gamma\right)\right)}+\left\Vert v_{T}\right\Vert _{V}\right),
\end{eqnarray}
for a positive constant $C$ that depends only on $\Gamma$, $(\mu_\alpha)_{\alpha\in \cA}$, $T$ and $\varphi$. 
\end{thm}

\begin{proof}[Proof of  Theorem~\ref{thm: regularity for v}]
It is enough  to prove estimate~\eqref{meilleure-esti-v} for $v_{n}$.

Multiplying \eqref{eq: approximate linear equation}
by $-\frac d {dt} y_{k}^{n}$, summing for $k=1,\dots,n$
and using \eqref{eq: formula of v_n}  leads to
\[
\int_{\Gamma}(\partial_{t}v_{n})^{2}\varphi dx-\int_{\Gamma}\mu  \partial v_{n}\partial\left(\partial_{t}v_{n}\varphi\right)dx=- \int_{\Gamma} h\partial_{t}v_{n}\varphi dx,
\]
hence
\begin{align*}
\int_{\Gamma} (\partial_{t}v_{n})^{2}\varphi dx-\int_{\Gamma}\mu  \partial_{t}\dfrac{(\partial v_{n})^{2}}{2}\varphi dx
-\int_{\Gamma}\mu  \partial v_{n}\partial_{t}v_{n}\partial\varphi dx=-\int_{\Gamma}  h\partial_{t}v_{n}\varphi dx.
\end{align*}
Multiplying by $e^{\lambda t}$ where $\lambda$ will chosen later, and taking the integral from $s$ to $T$, we obtain
\begin{eqnarray}
 && \int_{s}^{T}\int_{\Gamma}(\partial_{t}v_{n})^{2}e^{\lambda t}\varphi dxdt-\int_{\Gamma}\dfrac{\mu}{2}\left[(\partial v_{n}\left(T\right))^{2}e^{\lambda T}-(\partial v_{n}\left(s\right))^{2}e^{\lambda s}\right]\varphi dx\nonumber \\
 && +\lambda\int_{s}^{T}\int_{\Gamma}\dfrac{\mu}{2} (\partial v_{n})^{2} e^{\lambda t}\varphi dxdt-\int_{s}^{T}\int_{\Gamma}\mu\partial v_{n}\partial_{t}v_{n}e^{\lambda t}\partial\varphi dxdt\nonumber \\
&= & -\int_{s}^{T}\int_{\Gamma}h\partial_{t}v_{n}e^{\lambda t}\varphi dxdt\nonumber \\
&\le & \dfrac{1}{2}\int_{s}^{T}\int_{\Gamma}h^{2}e^{\lambda t}\varphi dxdt+\dfrac{1}{2}\int_{s}^{T}\int_{\Gamma}(\partial_{t}v_{n})^{2}e^{\lambda t}\varphi dxdt.\label{eq: estimate dt v_n}
\end{eqnarray}
Let us deal with the term $\int_{\Gamma}(\partial v_{n}\left(x,T\right))^{2}\varphi dx$. From \eqref{eq: aproximate terminal condtion v_n},
\begin{displaymath}
  \begin{split}
  \int_{\Gamma} \mu(\partial v_{n}\left(x,T\right))^{2}\varphi dx
&= \sum_{k=1}^n \lambda_k \left(\int_{\Gamma}  v_T \mathsf{v}_{k} dx\right)^2 \\
 &\le    \sum_{k=1}^\infty \lambda_k \left(\int_{\Gamma}  v_T \mathsf{v}_{k} dx\right)^2 =  \int_{\Gamma} \mu(\partial v_T\left(x\right))^{2}\varphi dx \\ & \le \overline{\mu} \int_{\Gamma} (\partial v_T\left(x\right))^{2}\varphi dx.
  \end{split}
\end{displaymath}
Then, choosing $\lambda =2\overline{\mu}^2||\partial \varphi||_{L^\infty(\Gamma)}^2/(\underline{\varphi}^2\underline{\mu})$, we obtain that 
\begin{equation}
\int_{\Gamma}2\mu(\partial v_{n}(x,s))^{2}\varphi dx + \int_{s}^{T}\int_{\Gamma}(\partial_t v_{n})^{2}\varphi dxdt\leq
  2e^{\lambda T}\overline{\varphi}\left( \left\Vert h\right\Vert _{L^{2}\left(  \Gamma \times(0,T)\right)}^{2}
  + \overline \mu \int_{\Gamma}(\partial v_{T})^{2}dx\right). 
\label{ineq:useful-formula-1}
\end{equation}

\emph{Estimate of $\partial v_{n}$ in $L^{\infty}\left(0,T;L^{2}\left(\Gamma\right)\right)$
and $\partial_{t}v_{n}$ in $L^{2}\left(  \Gamma \times(0,T)\right)$}.
From~\eqref{ineq:useful-formula-1}, it is straightforward to see that
\begin{eqnarray*}
  && \left\Vert \partial v_{n}\right\Vert _{L^{\infty}\left(0,T;L^{2}\left(\Gamma\right)\right)}
  +\left\Vert \partial_t v_{n}\right\Vert _{L^{2}\left(  \Gamma \times(0,T)\right)}
\leq C\left(\left\Vert h\right\Vert _{L^{2}\left(  \Gamma \times(0,T)\right)}+\left\Vert\partial  v_{T}\right\Vert _{L^{2}\left(\Gamma\right)}\right)
\end{eqnarray*}
for some constant $C$ depending only on $\Gamma$, $\mu$, $T$ and $\varphi$.

\emph{Estimate of $\partial^2 v_{n}$ in $L^{2}\left(  \Gamma \times(0,T)\right)$}.
Finally, using the  PDE in \eqref{eq: linear 1},
we can see that $\partial^{2}v_{n}$ belongs to $L^{2}\left(  \Gamma \times(0,T)\right)$
and is bounded by $C\left(\left\Vert h\right\Vert _{L^{2}\left(  \Gamma \times(0,T)\right)}+\left\Vert v_{T}\right\Vert _{V}\right)$,
hence $v_n$ is bounded in $L^{2}\left(0,T;H^{2}\left(\Gamma\right)\right)$ by the same quantity.
The Kirchhoff conditions (which boil down to Neumann conditions at $\partial \Gamma$) are therefore satisfied in a strong sense for almost all $t$.

Using \cite[Theorem 3.1]{lm68} (or a similar argument as \cite[pages 287-288]{Evans2010}), we see that  $v$ in $C([0,T]; V )$. 

\end{proof}

\section{The Fokker-Planck equation}
\label{sec:fokk-planck-equat}
This paragraph is devoted to a boundary value problem including
a Fokker-Planck equation
\begin{equation}
\begin{cases}
\partial_{t}m-\mu_{\alpha}\partial^{2}m-\partial\left(bm\right)=0, & \text{in }\left(\Gamma_{\alpha}\backslash\mathcal{V}\right)\times\left(0,T\right),~\alpha\in\mathcal{A},\\
\dfrac{m|_{\Gamma_{\alpha}}\left(\nu_{i},t\right)}{\gamma_{i\alpha}}=\dfrac{m|_{\Gamma_{\beta}}\left(\nu_{i},t\right)}{\gamma_{i\beta}}, &  t\in\left(0,T\right),~\alpha,\beta\in\mathcal{A}_i,~ \nu_{i}\in\mathcal{V}\backslash\partial\Gamma,\\
\displaystyle{\sum_{\alpha\in\mathcal{A}_{i}}\mu_{\alpha}\partial_{\alpha}m\left(\nu_{i},t\right)
  + n_{i\alpha}b\left(\nu_{i},t\right)m|_{\Gamma_{\alpha}}\left(\nu_{i},t\right)=0,} & t\in\left(0,T\right),~\nu_{i}\in\mathcal{V},\\
m\left(x,0\right)=m_{0}\left(x\right), & x\in\Gamma,
\end{cases}\label{eq: Fokker-Planck}
\end{equation}
where $b\in PC\left(\Gamma\times\left[0,T\right]\right)$
and $m_{0}\in L^2(\Gamma)$.
\begin{defn}
A weak solution of \eqref{eq: Fokker-Planck} is a function $ m\in L^{2}\left(0,T;W\right)\cap C([0,T]; L^2(\Gamma))$ such that $\partial_{t}m\in L^{2}\left(0,T;V'\right)$ and
\begin{equation}
\begin{cases}
\ds \langle \partial_{t}m,v\rangle _{V', V}+\mathscr{A}\left(m,v\right)=0\quad\text{for all }v\in V\text{ and a.e. }t\in (0,T),\\
m\left(\cdot,0\right)=m_{0},
\end{cases}\label{eq: linear 1 weak form-1}
\end{equation}
where $\mathscr{A}:W\times V\rightarrow\mathbb{R}$ is the bilinear form
\[
\mathscr{A}\left(v,w\right)=\int_{\Gamma}\mu\partial m\partial vdx+\int_{\Gamma}bm\partial vdx.
\]
\end{defn}

Using similar arguments as in Section~\ref{subsec: The-linear-parabolic}, in particular
a Galerkin method, we obtain the following result, the proof of which is omitted.


\begin{thm}
\label{thm: existence and uniqueness FK}
If $ b\in L^\infty (\Gamma \times (0,T))$ and $m_0\in L^2(\Gamma)$,  there exists a unique 
function $ m\in L^{2}\left(0,T;W\right)\cap C([0,T]; L^2(\Gamma))$ such that $\partial_{t}m\in L^{2}\left(0,T;V'\right)$ and (\ref{eq: linear 1 weak form-1}). Moreover, there exists a constant $C$ which depends on  $(\mu_\alpha)_{\alpha\in \cA}$, $\left\Vert b\right\Vert _{\infty}$, $T$ and $\varphi$, such that 
\begin{equation}\label{eq:22}
\left\Vert m\right\Vert _{L^{2}\left(0,T;W\right)}  
+\left\Vert m\right\Vert _{L^{\infty}\left(0,T; L^2\left(\Gamma\right)\right)}+
\left\Vert \partial_{t}m\right\Vert _{L^{2}\left(0,T;V'\right)}\le C\left\Vert m_{0}\right\Vert _{L^2(\Gamma)}.
\end{equation}
\end{thm}

\begin{rem}
If $m_0\in \mathcal{M}$, which will be the case when solving the MFG system~\eqref{eq: MFG system},
then $m(\cdot, t)\in \mathcal{M}$ for all $t\in [0,T]$. Indeed, 
we use $v\equiv1\in V$ as a test-function for~\eqref{eq: Fokker-Planck}. 
Since $\partial v=0$, integrating~\eqref{eq: linear 1 weak form-1} from $0$ to $t$, we get
$ \int_{0}^{t} \int_{\Gamma} \partial_t m(x,s) dxds=0$.  
This implies that 
\[
\int_{\Gamma} m(x,t) dx = \int_{\Gamma} m_0 (x) dx =1,\quad\text{ for all } t\in(0,T].
\]
Setting $m^-=-\mathds{1}_{\left\{m< 0\right\}} m$, we can also use  $v= \varphi^{-1} m^- e^{-\lambda t}$ as a  test-function 
 for $\lambda\in \R_+$. 
Indeed, the latter function  belongs to $L^2(0,T;V)$. Taking $\lambda$ large enough and using similar arguments as for the energy estimate (\ref{eq:22}) yield that $ m^-=0$, i.e., $m\ge 0$.
\end{rem}
We end this section by stating a stability result, which will be useful in the proof of
the main Theorem.

\begin{lem}\label{lem: stability FK}
Let $m_{0\varepsilon},b_{\varepsilon}$ be sequences of functions
satisfying
\[
m_{0\varepsilon}\longrightarrow m_{0}\text{ in }L^{2}\left(\Gamma\right),\qquad b_{\varepsilon}\longrightarrow b\text{ in }
L^{2}\left(\Gamma \times (0,T)\right),
\]
and for some positive number $K$ independent of $\varepsilon$, 
$\|b\|_{L^{\infty}\left(\Gamma \times (0,T)\right)}\le K$, $\|b_{\varepsilon}\|_{L^{\infty}\left(\Gamma \times (0,T)\right)}\le K$.
\\
Let $ m_{\varepsilon} $  (respectively  $m$) be the solution of   (\ref{eq: linear 1 weak form-1}) corresponding to the datum $m_{0\varepsilon}$ (resp. $m_0$)
and the coefficient $b_{\varepsilon}$ (resp. $b$).
The sequence  $(m_{\varepsilon})$
converges to $m$ in $L^{2}\left(0,T;W\right)\cap L^\infty\left(0,T;L^2(\Gamma)\right)$, and the sequence $
 (\partial_t m_{\varepsilon})$ converges to $(\partial_t m) $ in $L^{2}\left(0,T;V'\right)$. 
\end{lem}

\begin{proof}[Proof of Lemma~\ref{lem: stability FK}]
  Taking $\left(m_{\varepsilon}-m\right)e^{-\lambda t}\varphi^{-1}$ as a test-function in  the versions of (\ref{eq: linear 1 weak form-1}) satisfied by 
$m_{\varepsilon}$ and $m$, subtracting,  we obtain that
\begin{align}
 & \int_{\Gamma}\left[\dfrac{1}{2}\partial_{t}\left(\left(m_{\varepsilon}-m\right)^{2}e^{-\lambda t}\right)
+\dfrac{\lambda}{2}\left(m_{\varepsilon}-m\right)^{2}e^{-\lambda t}\right]\varphi^{-1}dx+
\int_{\Gamma}\mu (\partial\left(m_{\varepsilon}-m\right))^{2}e^{-\lambda t}\varphi^{-1}dx\nonumber \\
 & +\int_{\Gamma}\mu \left(m_{\varepsilon}-m\right) \partial\left(m_{\varepsilon}-m\right)e^{-\lambda t}\partial(\varphi^{-1})dx+\int_{\Gamma}\left(b_{\varepsilon}m_{\varepsilon}-bm\right)\partial\left(m_{\varepsilon}-m\right)e^{-\lambda t}\varphi^{-1}dx\nonumber \\
 & +\int_{\Gamma}\left(b_{\varepsilon}m_{\varepsilon}-bm\right)
\left(m_{\varepsilon}-m\right)e^{-\lambda t}\partial(\varphi^{-1})dx=0.\label{eq: stability for FK}
\end{align}
There exists a positive constant $K$ such that  $\left\Vert b_{\varepsilon}\right\Vert _{\infty},\left\Vert b\right\Vert _{\infty}\le K$ for all $\varepsilon$. Hence, there exists a positive constant $C$ (in fact it varies from one line to the other in what follows) such that
\begin{eqnarray*}
  &&  \int_{\Gamma}\left[\dfrac{1}{2}\partial_{t}\left(\left(m_{\varepsilon}-m\right)^{2}e^{-\lambda t}\right)+\dfrac{\lambda}{2}\left(m_{\varepsilon}-m\right)^{2}e^{-\lambda t}\right]\varphi^{-1}dx+\int_{\Gamma}\mu (\partial\left(m_{\varepsilon}-m\right))^{2}e^{-\lambda t}\varphi^{-1}dx\nonumber\\
  &\leq&
  C \int_{\Gamma} \left(\left|m_{\varepsilon}-m\right|^{2}
  +\left|m_{\varepsilon}-m\right|\left|\partial\left(m_{\varepsilon}-m\right)\right|
  + |m|\left|b_{\varepsilon}-b\right|\left(
  \left|\partial\left(m_{\varepsilon}-m\right)\right| +|m_{\varepsilon}-m|\right)\right)e^{-\lambda t}\varphi^{-1} dx\\
  &\leq&
  C \int_{\Gamma} \left(\left|m_{\varepsilon}-m\right|^{2}
  + \left|b_{\varepsilon}-b\right|^{2}  m^2 \right)e^{-\lambda t}\varphi^{-1} dx
  + \int_{\Gamma}\dfrac{\mu}{2} (\partial\left(m_{\varepsilon}-m\right))^{2}e^{-\lambda t}\varphi^{-1}dx.
\end{eqnarray*}

The assumptions on the coefficents $b_\varepsilon$ and $b$ imply in fact that 
$b_\varepsilon\to b$ in $L^p(\Gamma \times (0,T))$ for all $1\le p<\infty$.
On the other hand, we  know that $m\in L^q(\Gamma \times (0,T))$ for all $1\le q<\infty$.
From the latter  observation with $p=q=4$, we see that the quantity 
 $ \int_0^T \int_{\Gamma}   \left( \left|b_{\varepsilon}-b\right|^{2}  m^2 \right)e^{-\lambda t}\varphi^{-1} dx dt$
tends to $0$ as $\varepsilon\to 0$ uniformly in $\lambda>0$. We write 
\[\ds \int_0^T \int_{\Gamma}   \left( \left|b_{\varepsilon}-b\right|^{2}  m^2 \right)e^{-\lambda t}\varphi^{-1} dx dt =o_{\varepsilon}(1) .\]
 Choosing $\lambda$ large enough and integrating the latter inequality from $0$ to $t\in [0,T]$, we obtain
\[
\left\Vert m_{\varepsilon}-m\right\Vert _{L^{2}\left(0,T;W\right)}  + \left\Vert m_{\varepsilon}-m\right\Vert_{L^\infty\left(0,T;L^2(\Gamma)\right)}\le
o_\varepsilon(1)+
C \left\Vert m_{0\varepsilon}-m_{0}\right\Vert _{L^{2}\left(\Gamma\right)}.
\]
Subtracting the two versions of (\ref{eq: linear 1 weak form-1}) and using the latter estimate also yields
\[
\left\Vert \partial_t m_{\varepsilon}-\partial_t m\right\Vert _{L^{2}\left(0,T;V'\right)} 
\le o_\varepsilon(1)+
C\left\Vert m_{0\varepsilon}-m_{0}\right\Vert _{L^{2}\left(\Gamma\right)},
\]
which achieves the proof.
\end{proof}

\section{The Hamilton-Jacobi equation}
\label{sec:hamilt-jacobi-equat}
This section is devoted to the following boundary value problem including
a Hamilton-Jacobi equation
\begin{equation}
\begin{cases}
-\partial_{t}v-\mu_{\alpha}\partial^{2}v+H\left(x,\partial v\right)=f, & \text{in }\left(\Gamma_{\alpha}\backslash\mathcal{V}\right)\times\left(0,T\right),~\alpha\in \cA,\\
v|_{\Gamma_{\alpha}}\left(\nu_{i},t\right)=v|_{\Gamma_{\beta}}\left(\nu_{i},t\right) & t\in\left(0,T\right),~\alpha,\beta\in\mathcal{A}_{i},~\nu_{i}\in\mathcal{V},\\
\ds\sum_{\alpha\in\mathcal{A}_{i}}\gamma_{i\alpha}\mu_{\alpha}\partial_{\alpha}v\left(\nu_{i},t\right)=0, & t\in\left(0,T\right),~\nu_{i}\in\mathcal{V},\\
v\left(x,T\right)=v_{T}\left(x\right), & x\in\Gamma,
\end{cases}\label{eq: H-J}
\end{equation}
where $f\in L^{2}\left(\Gamma \times (0,T)\right)$, $v_T\in V$ and
the Hamiltonian $H:\Gamma\times\mathbb{R}\rightarrow\mathbb{R}$
satisfies the running assumptions (H).

\begin{defn} For $f\in L^{2}\left(\Gamma \times (0,T)\right)$ and $v_T \in V$,
  a weak solution of \eqref{eq: H-J} is a function $v\in L^{2}\left(0,T;H^{2}\left(\Gamma\right)\right)\cap C([0,T];V)$ such that $\partial_{t}v\in L^{2}\left(\Gamma \times (0,T)\right)$ and
  \begin{eqnarray}
\label{eq:23}    
  \left\{
  \begin{array}{l}
\ds\int_{\Gamma}\left(-\partial_{t}vw+\mu\partial v\partial w+H\left(x,\partial v\right)w\right)dx=\int_{\Gamma}fwdx
\quad\text{for all $w\in W,$  a.a. $t\in\left(0,T\right),$}\\
v\left(x,T\right)=v_{T}(x).
\end{array}
  \right.
\end{eqnarray}
\end{defn}

We start by proving existence and uniqueness of a weak solution for~\eqref{eq: H-J}.
 Next, further regularity for the solution will be  obtained  under stronger assumptions.

\subsection{Existence and uniqueness for the Hamilton-Jacobi equation}

\begin{thm}
\label{thm: existence and uniqueness of HJ equation}
Under the running assumptions (H),  if $f\in L^{2}\left(\Gamma\times (0,T)\right)$, then the boundary value
problem~\eqref{eq: H-J} has a unique weak solution.
\end{thm}

Uniqueness is a direct consequence of the following proposition.

\begin{prop}
  \label{thm: comparison principle HJ} (Comparison principle)
Under the same assumptions as in Theorem \ref{thm: existence and uniqueness of HJ equation},
 let $v$ and $\hat{v}$ be respectively weak sub- and super-solution
of \eqref{eq: H-J}, i.e., $v,\hat{v}\in L^{2}\left(0,T;H^{2}\left(\Gamma\right)\right),
\partial_{t}v,\partial_{t}\hat{v}\in L^{2}\left(\Gamma \times (0,T)\right)$
such that 
\begin{eqnarray*}
\left\{
\begin{array}{ll}
\begin{array}{l}
  \ds\int_{\Gamma}\left(-\partial_{t}vw+\mu\partial v\partial w+H\left(x,\partial v\right)w\right)dx\le\int_{\Gamma}fwdx,\\[3mm]
  \ds\int_{\Gamma}\left(-\partial_{t}\hat{v}w+\mu\partial\hat{v}\partial w+H\left(x,\partial\hat{v}\right)w\right)dx\ge\int_{\Gamma}fwdx,
\end{array}
&
\text{for all $w\in W,$ $w\geq 0$, a.a. $t\in\left(0,T\right),$}\\
v\left(x,T\right)\le v_{T}(x)\le\hat{v}\left(x,T\right) \quad \hbox{ for a.a. }x\in \Gamma.
\end{array}  
\right.
\end{eqnarray*}
Then $v\le\hat{v}$ in $\Gamma\times\left(0,T\right)$.
\end{prop}

\begin{proof}[Proof of Proposition~\ref{thm: comparison principle HJ}]
Setting $\overline{v}=v-\hat{v}$, we have, for all $w\in W$ such that  $w\geq 0$ and for a.a $t\in (0,T)$:
\[\
\int_{\Gamma} -\partial_{t}\overline{v}w+\mu\partial\overline{v}\partial w+\left(H\left(x,\partial v\right)-H\left(x,\partial\hat{v}\right)\right) w dx\le0,
\]
and $\overline{v}\left(x,T\right)\le0$ for all $x\in\Gamma$. Set  
$\overline{v}^{+}= \overline{v}\; \mathds{1}_ {\{\overline{v}>0\}}$ 
and $w=\overline{v}^{+}e^{\lambda t}\varphi$. We have
\begin{align*}
  & -\int_{\Gamma}\partial_{t}\left(\dfrac{(\overline{v}^{+})^{2}}{2}e^{\lambda t}\right)\varphi dx
  +\int_{\Gamma}\dfrac{\lambda}{2}(\overline{v}^{+})^{2}e^{\lambda t}\varphi dx
  +\int_{\Gamma}\mu \partial\overline{v}^{+}\partial (\overline{v}^{+}\varphi ) e^{\lambda t}dx\\
 & +\int_{\Gamma}\left[H\left(x,\partial v\right)-H\left(x,\partial\hat{v}\right)\right]\overline{v}^{+}\varphi e^{\lambda t}dx=0.
\end{align*}
Integrating from $0$ to $T$, we get
\begin{align*}
  & \int_{\Gamma}\left(\dfrac{\overline{v}^{+}\left(0\right)^{2}}{2}-\dfrac{\overline{v}^{+}\left(T\right)^{2}}{2}
  e^{\lambda T}\right)\varphi dx
  +\int_{0}^{T}\int_{\Gamma}\dfrac{\lambda}{2}(\overline{v}^{+})^{2}e^{\lambda t}\varphi dxdt\nonumber \\
  & +\int_{0}^{T}\int_{\Gamma} \mu (\partial\overline{v}^{+})^{2}\varphi e^{\lambda t}dxdt
  +\int_{0}^{T}\int_{\Gamma} \mu \partial\overline{v}^{+}\overline{v}^{+}\partial\varphi e^{\lambda t}dxdt\nonumber \\
 & +\int_{0}^{T}\int_{\Gamma}\left[H\left(x,\partial v\right)-H\left(x,\partial\hat{v}\right)\right]\overline{v}^{+}\varphi e^{\lambda t}dxdt = 0. 
\end{align*}
From (\ref{eq:18}),
$\left|H\left(x,\partial v\right)-H\left(x,\partial\hat{v}\right)\right|\le C_0|\partial\overline{v}|$.
Hence,  since $\overline{v}^{+}\left(T\right)=0$ and
$|\partial\overline{v}|\overline{v}^{+}=|\partial\overline{v}^{+}|\overline{v}^{+}$ almost everywhere, we get
\begin{align}
  \int_{0}^{T}\int_{\Gamma}\left(\dfrac{\lambda}{2}(\overline{v}^{+})^{2}
  +\mu (\partial\overline{v}^{+})^{2}\right)e^{\lambda t}\varphi dxdt
  -\int_{0}^{T}\int_{\Gamma}\left(\mu \left|\partial\varphi\right|+C_0\varphi\right)
  |\partial\overline{v}^{+}|\overline{v}^{+}e^{\lambda t}dxdt\le0.
  \label{eq: contradiction comparison HJ}
\end{align}
For $\lambda$ large enough, the first term in the left hand side is not smaller than the second term.
This implies that   $\overline{v}^{+} = 0$.
\end{proof}

Now we prove Theorem~\ref{thm: existence and uniqueness of HJ equation}.
We start with a bounded Hamiltonian  $H$.

\begin{proof}[Proof of existence in Theorem~\ref{thm: existence and uniqueness of HJ equation} when $H$ is bounded by $C_{H}$]
Take $\overline{v}\in L^{2}\left(0,T;V\right)$ and $f\in L^{2}\left(\Gamma\times (0,T)\right)$.
From Theorem~\ref{thm: existence and uniqueness linear equation}
and Theorem~\ref{thm: regularity for v} with $h=f-H\left(x,\partial\overline{v}\right)$ and $v_{T}\in V$, 
the following boundary value problem
\begin{equation}
\begin{cases}
-\partial_{t}v-\mu_{\alpha}\partial^{2}v=f-H\left(x,\partial\overline{v}\right), & \text{in }\left(\Gamma_{\alpha}\backslash\mathcal{V}\right)\times\left(0,T\right),~\alpha\in \cA,\\
v|_{\Gamma_{\alpha}}\left(\nu_{i},t\right)=v|_{\Gamma_{\beta}}\left(\nu_{i},t\right), & t\in\left(0,T\right),~\alpha,\beta\in\mathcal{A}_{i},~\nu_i\in\mathcal{A}_i,\\
\sum_{\alpha\in\mathcal{A}_{i}}\gamma_{i\alpha}\mu_{\alpha}\partial_{\alpha}v\left(\nu_{i},t\right)=0, & t\in\times\left(0,T\right),~\nu_{i}\in\mathcal{V},\\
v\left(x,T\right)=v_{T}\left(x\right), & x\in\Gamma,
\end{cases}\label{eq: fixed point}
\end{equation}
has a unique weak solution $v\in L^{2}\left(0,T;H^{2}\left(\Gamma\right)\right)\cap C([0,T];V) \cap W^{1,2}\left(0,T;L^{2}\left(\Gamma\right)\right)$.
This allows us to define the map $T$:
\begin{align*}
T:L^{2}\left(0,T;V\right) & \longrightarrow L^{2}\left(0,T;V\right),\\
\overline{v} & \longmapsto v.
\end{align*}
From (\ref{eq:18}),  $\overline{v}\longmapsto H\left(x,\partial\overline{v}\right)$
is continuous from $L^{2}\left(0,T;V\right)$ into $L^{2}\left(\Gamma\times(0,T)\right)$.
Using again Theorem \ref{thm: regularity for v}, we have that $T$
is continuous from $L^{2}\left(0,T;V\right)$ to $L^{2}\left(0,T;V\right)$.
Moreover, there exists a constant $C$ depending only on $C_{H}$,$\Gamma$,$(\mu_{\alpha})_{\alpha\in \cA}$, $f$, $T$, $\varphi$
and $v_{T}$ such that
\begin{equation}
\left\Vert \partial_{t}v\right\Vert _{L^{2}\left(\Gamma\times(0,T)\right)}+\left\Vert v\right\Vert _{L^{2}\left(0,T;H^{2}\left(\Gamma\right)\right)}\le C.\label{eq: uniformly bounded of truncation}
\end{equation}
Therefore, from Aubin-Lions theorem (see Lemma \ref{aubin-lions-lem}), we obtain that
$T\left(L^{2}\left(0,T;V\right)\right)$ is relatively compact in $L^{2}\left(0,T;V\right)$.
By  Schauder fixed point theorem, see \cite[Corollary 11.2]{GT2001},
$T$ admits a fixed point which is a weak solution of \eqref{eq: H-J}.
\end{proof}

\begin{proof}[Proof of existence in Theorem~\ref{thm: existence and uniqueness of HJ equation} in the general case]
Now we truncate the Hamiltonian as follows
\[
H_{n}\left(x,p\right)=\begin{cases}
H\left(x,p\right) & \text{if }\left|p\right|\le n,\\
H\left(x,\dfrac{p}{\left|p\right|}n\right) & \text{if }\left|p\right|>n.
\end{cases}
\]
From the previous proof for bounded Hamiltonians, for all $n$, there exists a solution
$v_{n}\in L^{2}\left(0,T;H^{2}\left(\Gamma\right)\right)\cap C([0,T];V)\cap W^{1,2}\left(0,T;L^{2}\left(\Gamma\right)\right)$
of \eqref{eq: H-J}, where $H$ is replaced by $H_{n}$.
We propose to send $n$ to $+\infty$ and to show a subsequence of
$\left\{ v_{n}\right\} $ converges to a solution of \eqref{eq: H-J}.
Hence, we need some uniform estimates for $\left\{ v_{n}\right\} $.
As in the proof of Proposition~\ref{thm: comparison principle HJ}, using $-v_{n}e^{\lambda t}\varphi$ as a test-function, integrating from $0$ to $T$ and noticing that $H$ is sublinear, see (\ref{eq:17}), we obtain
\begin{eqnarray*}
  && \int_{\Gamma}\left[\dfrac{v_{n}^{2}\left(x,0\right)}{2}-\dfrac{v_{n}^{2}\left(x,T\right)}{2}e^{\lambda T}\right]\varphi dx+\int_{0}^{T}\int_{\Gamma}\left[\frac{\lambda}{2}v_{n}^{2}e^{\lambda t}\varphi +\mu\left|\partial v_{n}\right|^{2}e^{\lambda t}\varphi+\mu\partial v_{n}v_{n}e^{\lambda t}\partial\varphi \right] dxdt\nonumber\\
  &=&
   - \int_{0}^{T}\int_{\Gamma}H_{n}\left(x,\partial v_{n}\right)v_{n}e^{\lambda t}\varphi dxdt
  + \int_{0}^{T}\int_{\Gamma}fv_{n}e^{\lambda t}\varphi dxdt\\
  &\leq&
    C_0\int_{0}^{T}\int_{\Gamma}\left(1+\left|\partial v_{n}\right|\right)\left|v_{n}\right|e^{\lambda t}\varphi dxdt
  + \dfrac{1}{2}\int_{0}^{T}\int_{\Gamma}f^2e^{\lambda t}\varphi dxdt 
  + \dfrac{1}{2}\int_{0}^{T}\int_{\Gamma}v^{2}_{n}e^{\lambda t}\varphi dxdt.
\end{eqnarray*}
In the following lines, the constant $C$ above will vary from
line to line and will  depend only on $(\mu_\alpha)_{\alpha\in \cA}, C_H, T$ and $\varphi$.
Taking $\lambda$ large enough leads to the following estimate:
\begin{equation}
\left\Vert v_{n}\right\Vert _{L^{2}\left(0,T;V\right)}\le C\left(\left\Vert f\right\Vert _{L^{2}\left(0,T;L^{2}\left(\Gamma\right)\right)}+\left\Vert v_{T}\right\Vert _{L^{2}\left(\Gamma\right)}+1\right),\label{eq: estimate truncate v_n}
\end{equation}
and thus, from (\ref{eq:17}) again, we also obtain
\begin{align*}
\int_{0}^{T}\int_{\Gamma}\left|H_{n}\left(x,\partial v_{n}\right)\right|^{2}dxdt & \le\int_{0}^{T}\int_{\Gamma}C_0^{2}\left(\left|\partial v_{n}\right|+1\right)^{2}dxdt \le\int_{0}^{T}\int_{\Gamma}2C_0^{2}\left(\left|\partial v_{n}\right|^{2}+1\right)dxdt\\
& \le C\left(\left\Vert f\right\Vert _{L^{2}\left(0,T;L^{2}\left(\Gamma\right)\right)}^2
+\left\Vert v_{T}\right\Vert _{L^{2}\left(\Gamma\right)}^2+1\right).
\end{align*}

Therefore, $\left\{ H_{n}\left(x,\partial v_{n}\right)-f\right\} $
is uniformly bounded in $L^{2}\left(0,T;L^{2}\left(\Gamma\right)\right)$.
From Theorem~\ref{thm: regularity for v}, we obtain that $\left( v_{n}\right) _{n\in\mathbb{N}}$
is uniformly bounded in $L^{2}\left(0,T;H^{2}\left(\Gamma\right)\right)\cap C([0,T];V)\cap W^{1,2}\left(0,T;L^{2}\left(\Gamma\right)\right)$.
By the Aubin-Lions theorem (see Lemma \ref{aubin-lions-lem}), $(v_n)_n$ is relatively compact in $L^{2}\left(0,T;V\right)$ 
(and bounded in $C\left([0,T];V\right)$).
 Hence, up
to the extraction of a subsequence, there exists $v\in L^{2}\left(0,T;V\right)\cap W^{1,2}\left(0,T;L^{2}\left(\Gamma\right)\right)$
such that
\begin{equation}
v_{n}  \rightarrow v,\quad\text{in }L^{2}\left(0,T;V\right)\quad  \hbox{(strongly)},\quad \quad \partial_{t}v_{n} \rightharpoonup\partial_{t}v,\quad\text{in } L^{2}(\Gamma\times(0,T)) \quad  \hbox{(weakly)}.
\end{equation}
Hence, $H_{n}\left(x,\partial v_{n}\right)\to H\left(x,\partial v\right)$ 
a.e. in $\Gamma\times (0,T)$. 
Note also that we can apply Lebesgue dominated convergence theorem to 
$H_{n}\left(x,\partial v_{n}\right)$  because $H_{n}\left(x,\partial v_{n}\right) \le H \left(x,\partial v_{n}\right)
\le C_0 (1+ |\partial v_n|)$. Therefore, $ H_{n}\left(x,\partial v_{n}\right)\to H\left(x,\partial v\right)$ in $L^2 (\Gamma\times(0,T))$. Thus, it is possible to pass to the limit in the weak formulation satisfied by $v_n$ and obtain that 
for all $w\in W$, $\chi \in C_c (0,T)$,
\[\int_0^T \chi(t)\left( 
-\int_{\Gamma}\partial_{t}vwdx+\int_{\Gamma}\partial v\partial wdx+\int_{\Gamma}H\left(x,\partial v\right)wdx\right) dt=
\int_0^T \chi(t) \left(\int_{\Gamma}fwdx \right) dt.
\]
Therefore, $v$ satisfies the first line in~\eqref{eq:23}.

From Theorem~\ref{thm: existence and uniqueness linear equation},
$v_{n}\left(T\right)=v_{T}$ for all $n$.
Since for all $\alpha\in \cA$,  $(v_n)_n$ tends to $v$ in $L^{2}(\Gamma_\alpha \times (0,T))$ strongly and in
$W^{1,2}(\Gamma_\alpha \times (0,T))$ weakly,
 $v_n|_{\Gamma_\alpha \times \{t=T\}} $ converges to $v|_{\Gamma_\alpha \times \{t=T\}} $
in $L^2 (\Gamma_\alpha)$ strongly. Passing to the limit in the latter identity, we get the second
condition in~\eqref{eq:23}.
We have proven that $v$ is a weak solution of~\eqref{eq: H-J}. 
\end{proof}

We end the section with a stability result for the Hamilton-Jacobi equation.

\begin{lem}\label{lem: stability HJ}
Let $(v_{T\varepsilon})_\varepsilon,(f_{\varepsilon})_\varepsilon$ be sequences of functions
satisfying
\[
v_{T\varepsilon}\longrightarrow v_{T}\text{ in }V,\qquad f_{\varepsilon}\longrightarrow f\text{ in }L^{2}\left(\Gamma\times(0,T)\right).
\]
Let $ v_{\varepsilon} $ be the weak solution of \eqref{eq: H-J}
with data $v_{T\varepsilon},f_{\varepsilon}$, then $(v_{\varepsilon})_\varepsilon$
converges in $ L^{2}\left(0,T;H^2(\Gamma)\right)\cap C([0,T];V) \cap W^{1,2}\left(0,T;L^{2}\left(\Gamma\right)\right)$ to the weak solution $v$ of \eqref{eq: H-J} with data $v_{T},f$.
\end{lem}

\begin{proof}[Proof of Lemma~\ref{lem: stability HJ}]
Subtracting the two PDEs for $v_{\varepsilon}$ and $v$,
multiplying  by $\left(v_{\varepsilon}-v\right)e^{\lambda t}\varphi^{-1}$,
taking the integral on $\Gamma\times(0,T)$ and using similar computations as 
in the proof of Proposition~\ref{thm: comparison principle HJ}, we obtain
\[
\left\Vert v_{\varepsilon}-v\right\Vert _{L^{2}\left(0,T;V\right)}\le C\left(\left\Vert f_{\varepsilon}-f\right\Vert _{L^2\left(\Gamma\times(0,T)\right)}+\left\Vert v_{T\varepsilon}-v_{T}\right\Vert _{L^{2}\left(\Gamma\right)}\right),
\]
for $\lambda$ large enough and $C$ independent of $\varepsilon$.  This proves the convergence of $v_\epsilon$ to $v$ in 
$L^{2}\left(0,T;V\right)$. Then, the convergence in  $ L^{2}\left(0,T;H^2(\Gamma)\right)\cap C([0,T];V) \cap W^{1,2}\left(0,T;L^{2}\left(\Gamma\right)\right)$ results from the assumption that $H$ is Lipschitz with respect to its second argument,  and from stability results for the linear boundary value problem (\ref{eq: linear 1}) which are obtained with similar arguments as in the proof of Theorem~\ref{thm: regularity for v}.
\end{proof}

\subsection{Regularity for the Hamilton-Jacobi equation}

In this section, we prove further regularity for the solution of~\eqref{eq: H-J}.

\begin{thm}\label{thm: H-J regularity}
We suppose that the assumptions of Theorem~\ref{thm: existence and uniqueness of HJ equation}
hold and that, in addition, $v_T\in H^2(\Gamma)$ satisfies the Kirchhoff  conditions given by the third equation in (\ref{eq: H-J}),   $f\in  PC(\Gamma\times [0,T])\cap L^2(0,T; H^1_b (\Omega))$ 
and  $\partial_t f\in L^2(0,T;  H^1_b(\Gamma))$.

Then, the unique solution $v$ of \eqref{eq: H-J}
satisfies $v\in L^{2}\left(0,T;H^{3}\left(\Gamma\right)\right)$ and
$\partial_{t}v\in L^{2}\left(0,T;H^{1}\left(\Gamma\right)\right)$.
Moreover, there exists a constant $C$ depending only on $\left\Vert v_{T}\right\Vert _{H^{2}\left(\Gamma\right)}$,
$(\mu_\alpha)_{\alpha\in \cA} $, $H$ and 
$f$ such that
\begin{eqnarray}\label{energy-est-v}
&& \left\Vert v\right\Vert _{L^{2}\left(0,T;H^{3}\left(\Gamma\right)\right)}+\left\Vert \partial_{t}v\right\Vert _{L^{2}\left(0,T;H^{1}\left(\Gamma\right)\right)}\le C.
\end{eqnarray}
If, in addition,  there exists $\eta\in (0,1)$ such  that $v_T\in C^{2+\eta}(\Gamma)$
then  there exists $\tau\in (0,1)$ such $v\in C^{2+\tau,1+\frac{\tau}{2}}(\Gamma\times [0,T])$, and $v$ is a classical solution of (\ref{eq: H-J}).
\end{thm}

The main idea to prove Theorem~\ref{thm: H-J regularity} is to differentiate~\eqref{eq: H-J} with respect to the space
variable and to prove some regularity properties for the derived equation. Let us explain formally our
method. Assuming the solution $v$ of~\eqref{eq: H-J} is in $C^{2,1}(\Gamma\times(0,T))$
and taking the space-derivative of~\eqref{eq: H-J}
on $\left(\Gamma_{\alpha}\backslash\mathcal{V}\right)\times\left(0,T\right)$, we have
\[
-\partial_t \partial v - \mu_\alpha \partial^3 v + \partial \left( H (x,\partial v) \right) = \partial f . 
\]
Therefore, $u=\partial v$ satisfies the following PDE
\[
-\partial_{t}u-\mu_{\alpha}\partial^{2}u+\partial \left( H \left(x,u\right)\right)=\partial f,
\]
with terminal condition $u(x,T)=\partial v_T(x)$.
From the Kirchhoff conditions in~\eqref{eq: H-J} and Remark~\ref{rem:new network}, we
obtain a condition for $u$ of Dirichlet type, namely 
\[
\sum_{\alpha \in \mathcal{A}_i}\mu_{\alpha}\gamma_{i \alpha}  n_{i\alpha} u|_{\Gamma_{\alpha}} (\nu_i,t)=0,\quad t \in (0,T),~\nu_i \in \mathcal{V}.
\]
Note that the latter condition is an homogeneous Dirichlet condition at the boundary vertices of $\Gamma$.

Now, by extending continuously the PDEs in~\eqref{eq: H-J} until the vertex $\nu_i$
in the branchs $\Gamma_{\alpha}$ and $\Gamma_{\beta}$, $\alpha,\beta \in \mathcal{A}_i$,
and using the continuity condition in~\eqref{eq: H-J}, one gets
\begin{eqnarray*}
-\mu_{\alpha} \partial^{2} v|_{\Gamma_{\alpha}} + H^{\alpha}\left(\nu_i,\partial v|_{\Gamma_{\alpha}}(\nu_i,t)\right)-f|_{\Gamma_{\alpha}} (\nu_i,t) = -\mu_{\beta} \partial^{2} v|_{\Gamma_{\beta}} + H^\beta (\nu_i,\partial v|_{\Gamma_{\beta}}(\nu_i,t) )-f|_{\Gamma_{\beta}} (\nu_i,t).
\end{eqnarray*}
This gives a second transmission condition  for $u$ at $\nu_i\in \cV\backslash {\partial \Gamma}$  of Robin type, namely
\begin{equation}
  \label{eq:25}
  \begin{split}
&
\mu_{\alpha}\partial u|_{\Gamma_{\alpha}}\left(\nu_{i},t\right)-H^\alpha(\nu_{i},u|_{\Gamma_{\alpha}}(\nu_{i},t))+f|_{\Gamma_{\alpha}}\left(\nu_{i},t\right)
\\ =&\mu_{\beta}\partial u|_{\Gamma_{\beta}}\left(\nu_{i},t\right)-H^\beta(\nu_{i},u|_{\Gamma_{\beta}}(\nu_{i},t))+f|_{\Gamma_{\beta}}(\nu_{i},t),    
  \end{split}
\end{equation}
which is equivalent to 
\begin{equation}
  \label{eq:26}
  \begin{split}
&\mu_{\alpha}  n_{i\alpha} \partial_\alpha u \left(\nu_{i},t\right)-H^\alpha(\nu_{i},u|_{\Gamma_{\alpha}}(\nu_{i},t))+f|_{\Gamma_{\alpha}}\left(\nu_{i},t\right)
\\ =& \mu_{\beta} n_{i\beta} \partial_\beta  u \left(\nu_{i},t\right)-H^\beta(\nu_{i},u|_{\Gamma_{\beta}}(\nu_{i},t))+f|_{\Gamma_{\beta}}(\nu_{i},t).
  \end{split}
\end{equation}

Hence, we shall study the following nonlinear boundary value problem for $u=\partial v$,
\begin{equation}
\begin{cases}
-\partial_{t}u-\mu_{\alpha}\partial^{2}u+\partial \left(H\left(x,u\right)\right)=\partial f\left(x,t\right), & (x,t)\in \left(\Gamma_{\alpha}\backslash\mathcal{V}\right)\times\left(0,T\right),~\alpha\in \cA,\\
{\displaystyle \sum_{\alpha\in\mathcal{A}_{i}}\gamma_{i\alpha}\mu_{\alpha} n_{i\alpha}   u|_{\Gamma_{\alpha}}\left(\nu_{i},t\right)=0,} & t\in\left(0,T\right),~\nu_{i}\in\mathcal{V},\\
\mu_{\alpha}n_{i\alpha} \partial_\alpha u\left(\nu_{i},t\right)-H^\alpha(\nu_{i},u|_{\Gamma_{\alpha}}(\nu_{i},t))
+f|_{\Gamma_{\alpha}}\left(\nu_{i},t\right)\\
\hspace*{0.4cm}=\mu_{\beta}  n_{i\beta} \partial_\beta  u \left(\nu_{i},t\right)
-H^\beta(\nu_{i},u|_{\Gamma_{\beta}}(\nu_{i},t))+f|_{\Gamma_{\beta}}\left(\nu_{i},t\right), & t\in\left(0,T\right),~\alpha,\beta\in\mathcal{A}_{i},~\nu_{i}\in\mathcal{V}\backslash\partial\Gamma,\\
u\left(x,T\right)=u_{T}\left(x\right), & x\in\Gamma,
\end{cases}\label{eq: regularity for H-J}
\end{equation}
where $\partial f\in L^2\left(\Gamma\times(0,T)\right)$ and $u_{T}\in F$ defined in (\ref{eq: space U}) below.
Theorem~\ref{thm: H-J regularity} will follow by choosing  $u_T=\partial v_T$.

In order to define the weak solutions of~\eqref{eq: regularity for H-J}, we need
the following subspaces of $H^{1}_{b}\left(\Gamma\right)$.

\begin{defn}
\label{def: functional spaces regularity} We define the Sobolev spaces
\begin{eqnarray}
  && F:  =\left\{u\in H_{b}^{1}\left(\Gamma\right)\text{ and }\sum_{\alpha\in\mathcal{A}_{i}}\gamma_{i\alpha}\mu_{\alpha} n_{i\alpha} u|_{\Gamma_{\alpha}}\left(\nu_{i}\right)=0\text{ for all }\nu_{i}\in\mathcal{V}\right\},\label{eq: space U}\\
  && E:=\left\{\mathsf{e}\in H_{b}^{1}\left(\Gamma\right)\text{ and }\sum_{\alpha\in\mathcal{A}_{i}}  n_{i\alpha}\mathsf{e}|{}_{\Gamma_{\alpha}}\left(\nu_{i}\right)=0\text{ for all }\nu_{i}\in\mathcal{V}\right\}.\label{eq: space E}
\end{eqnarray}
\end{defn}

\begin{defn}
  \label{sec:regul-hamilt-jacobi}
Let the function $\psi$ be defined as follows:
\begin{equation}
\begin{cases}
\psi_{\alpha}\text{ is affine on }\left(0,\ell_{\alpha}\right),\\
\psi|_{\Gamma_{\alpha}}\left(\nu_{i}\right)=\mu_{\alpha}\gamma_{i\alpha},\text{ if }\nu_i\in \cV\setminus \partial \Gamma,\;\alpha\in\mathcal{A}_{i},\\
\psi\text{ is constant on the edges \ensuremath{\Gamma_{\alpha}} which touch the boundary of \ensuremath{\Gamma}}.
\end{cases}\label{eq:27}
\end{equation}
Note that $\psi$ is positive and bounded. 
The map $f\longmapsto f\psi$
is an isomorphism  from $F$ onto $E$.
\end{defn}

\begin{defn}\label{def882}
A weak solution of~\eqref{eq: regularity for H-J} is a function $u\in  L^{2}\left(0,T;F\right)$ such that 
$ \partial_{t}u\in L^{2}\left(0,T;E'\right)$, and 
\begin{eqnarray}
  \label{eq:28}
\left\{
\begin{array}{l} 
\ds -\langle \partial_{t}u,\mathsf{e}\rangle_{E', E}  +\int_{\Gamma}\Bigl(\mu\partial u\partial\mathsf{e}- \left( H\left(x,u\right) \right) \partial \mathsf{e}\Bigr)dx = -\int_{\Gamma} f \partial\mathsf{e}dx,  \ \hbox{ for all } e\in E, \hbox{ a.a } t\in (0,T),\\
u\left(\cdot,T\right)=u_{T}. 
\end{array}
\right.
\end{eqnarray}
\end{defn}
\begin{rem}
  \label{sec:regul-hamilt-jacobi-1}
Note that if $u$ is regular enough, then~\eqref{eq:28} can also be written
 \begin{equation}
   \label{eq:30}
\begin{split}
& \ds  -\langle \partial_{t}u,\mathsf{e}\rangle_{E', E}  + \int_{\Gamma}\Bigl(\mu\partial u\partial\mathsf{e}+\partial \left( H\left(x,u\right) \right)  \mathsf{e}\Bigr)dx 
 -\sum_{i\in I}\sum_{\alpha\in\mathcal{A}_{i}}n_{i\alpha}\left[H^{{\alpha}}\left(\nu_{i},u|_{\Gamma_{\alpha}}\left(\nu_{i},t\right)\right)-f|_{\Gamma_{\alpha}}\left(\nu_{i},t\right)\right]\mathsf{e}|_{\Gamma_{\alpha}}\left(\nu_{i}\right)
\\ =&\int_{\Gamma} (\partial f) \mathsf{e}dx     \quad\quad \hbox{ for all } e\in E, \hbox{ a.a } t\in (0,T).
\end{split}  
 \end{equation}

\end{rem}

\begin{rem}\label{explication-def-weak-sol}
To explain formally the definition of weak solutions, let us use $\mathsf{e}\in E$ as a test-function in the PDE
in~\eqref{eq: regularity for H-J}. After an integration by parts, we get
\begin{eqnarray*}
  && \int_{\Gamma}\left(-\partial_{t}u\mathsf{e}+\mu\partial u\partial\mathsf{e}+\partial\left(H\left(x,u\right) \right)
 \mathsf{e}\right)dx
  - \sum_{i\in I}\sum_{\alpha\in\mathcal{A}_{i}} n_{i\alpha}\mu_\alpha\partial u|_{\Gamma_\alpha}(\nu_{i}, t) \mathsf{e}|_{\Gamma_{\alpha}}(\nu_{i})
=\int_\Gamma (\partial f) \mathsf{e}dx,
\end{eqnarray*}
where $n_{i\alpha}$ is defined in~\eqref{eq:1}. On the one hand, from
the second transmission condition, there exists a function $c_i: (0,T)\to \R$ such that
$\mu_{\alpha}  \partial u|_{\Gamma_\alpha} \left(\nu_{i},t\right)
-H^{{\alpha}}(\nu_{i},u|_{\Gamma_{\alpha}}(\nu_{i},t)) +f|_{\Gamma_{\alpha}}\left(\nu_{i},t\right)= c_i(t)$
for all $\alpha\in \mathcal{A}_{i}$. 
 It follows that
\begin{eqnarray*}
&& - \sum_{i\in I}\sum_{\alpha\in\mathcal{A}_{i}} n_{i\alpha}\mu_\alpha   \partial u|_{\Gamma_\alpha} (\nu_{i}, t) \mathsf{e}|_{\Gamma_{\alpha}}(\nu_{i})\\
  &=&  - \sum_{i\in I} c_i(t) \sum_{\alpha\in\mathcal{A}_{i}}n_{i\alpha} \mathsf{e}|_{\Gamma_{\alpha}}(\nu_{i})
  +  \sum_{i\in I}\sum_{\alpha\in\mathcal{A}_{i}}n_{i\alpha}
  \left[-H^{{\alpha}}(\nu_{i},u|_{\Gamma_{\alpha}}(\nu_{i},t)) +f|_{\Gamma_{\alpha}}\left(\nu_{i},t\right)\right]
  \mathsf{e}|_{\Gamma_{\alpha}}(\nu_{i})\\
  &=&
   \sum_{i\in I}\sum_{\alpha\in\mathcal{A}_{i}}n_{i\alpha}
  \left[-H^{{\alpha}}(\nu_{i},u|_{\Gamma_{\alpha}}(\nu_{i},t)) +f|_{\Gamma_{\alpha}}\left(\nu_{i},t\right)\right]
  \mathsf{e}|_{\Gamma_{\alpha}}(\nu_{i}),
\end{eqnarray*}
because $ \mathsf{e}\in E$. Then we may use the Remark \ref{sec:regul-hamilt-jacobi-1} and  obtain~\eqref{eq:28}.
\end{rem}  

We start by proving the following result about~\eqref{eq: regularity for H-J} and then
give the proof of Theorem~\ref{thm: H-J regularity}.

\begin{thm}\label{regularite-sys-derive}
Under the running assumptions,
if  $u_T\in F$, $f\in  C(\Gamma\times [0,T])\cap L^2 (0,T; H_b^1(\Gamma)$ 
and  $\partial_t f\in L^2(0,T; H^1_b(\Gamma))$,
then \eqref{eq: regularity for H-J} has a unique weak solution $u$. Moreover, there exists a constant
$C$ depending only on $\Gamma$, $T$, $\psi$, $\left\Vert u_{T}\right\Vert _{F}$,  $\left\Vert \partial f\right\Vert _{L^2\left(\Gamma\times(0,T)\right)}$,
 $\|f\|_{C(\Gamma\times[0,T])}$ 
and $\|\partial_t f \|_{L^2(0,T; H^1_b(\Gamma))  }$ 
 such that
\begin{eqnarray}\label{energ-est}
\left\Vert u\right\Vert _{L^{2}(0,T;H_{b}^{2}\left(\Gamma\right))}+\left\Vert u\right\Vert _{C\left( [0,T];F\right)}+
\left\Vert \partial_{t}u\right\Vert _{L^2\left(\Gamma\times(0,T)\right)}\le C.
\end{eqnarray}
\end{thm}

\begin{rem}\label{transmission-classical-sense}
Theorem \ref{regularite-sys-derive} implies that  $u (\cdot, t) \in C^{1}\left(\Gamma_\alpha\right)$ 
for all $\alpha\in \mathcal{A}$ for a.e. $t$. Hence, the transmission
conditions for $u$ hold in a classical sense for a.e. $t\in [0,T]$.
\end{rem}

We use the Galerkin's method to construct solutions of certain
finite-dimension approximations to~\eqref{eq: regularity for H-J}.

We notice first that the symmetric bilinear form $\widecheck \cB( u,v ):= \int_\Gamma \mu \psi^{-1} \partial u \partial v$ is 
such that $(u,v)\mapsto  (u,v)_{L^2 (\Gamma)}+ \widecheck \cB( u,v )$ is an inner product in $E$ equivalent to the standard inner product in $E$, namely $(u,v)_{E}= (u,v)_{L^2 (\Gamma)} + \int_\Gamma  \partial u \partial v$. Therefore, by standard Fredholm's theory, there exist 
\begin{itemize}
\item a non decreasing sequence  of nonnegative real numbers $(\lambda_k)_{k=1}^{\infty}$, that tends to $+\infty$ as $k\to \infty$
\item A Hilbert basis  $\left( \mathsf{e}_{k}\right) _{k=1} ^{\infty}$  of $L^2(\Gamma)$, which is also a 
 a total sequence   of $E$ (and orthogonal if $E$ is endowed with the scalar product  $  (u,v)_{L^2 (\Gamma)}+ \widecheck \cB( u,v )$),
\end{itemize}
such that 
\begin{equation}
\label{eq:31}
\widecheck \cB( \mathsf{e}_{k} ,e )= \lambda_k  (\mathsf{e}_{k} ,e)_{L^2 (\Gamma)} , \quad \hbox{for all } e\in E.
\end{equation}
Note that 
\[
 \int_{\Gamma}\mu\partial\mathsf{e}_{k}\partial\mathsf{e}_{\ell}\psi^{-1} dx=\begin{cases}
 \lambda_{k} & \text{if \ensuremath{k=\ell}},\\
 0 & \text{if \ensuremath{k\ne\ell.}}
\end{cases}
 \]
Note also that $\mathsf{e}_{k}$ is a weak solution of
\begin{equation}
\begin{cases}
-\mu_{\alpha}\partial\left(\psi^{-1}\partial\mathsf{e}_{k}\right)=\lambda_{k}\mathsf{e}_{k} & \text{in }\Gamma_{\alpha}\backslash\mathcal{V},\alpha\in\mathcal{A},\\
\dfrac{\partial_{\alpha}\mathsf{e}_{k}\left(\nu_{i}\right)}{\gamma_{i\alpha}}=\dfrac{\partial_{\beta}\mathsf{e}_{k}\left(\nu_{i}\right)}{\gamma_{i\beta}} & \text{for all }\alpha,\beta\in\mathcal{A}_{i},\\
\sum_{\alpha\in\mathcal{A}_{i}}  n_{i\alpha}e_{k}|_{\Gamma_{\alpha}}\left(\nu_{i}\right)=0 & \text{if }\nu_{i}\in\mathcal{V}.
\end{cases}\label{eq: eigenvalue problem for E}
\end{equation}
which implies that $\mathsf{e}_{k}|_{\Gamma_\alpha} \in C^2 (\Gamma_\alpha)$ for all $\alpha\in \cA$.

Finally,  the sequence $(\mathsf{f}_{k} )_{k=1}^{\infty}$ given by $\mathsf{f}_{k}= \psi^{-1} \mathsf{e}_{k}$ is a total family in $F$ (but is not orthogonal).

\begin{lem}
\label{lem: Galerkin's method for regularity HJ} Under the assumptions made in Theorem \ref{regularite-sys-derive},
for any positive integer $n$, there
exist $n$ absolutely continuous functions $y_{k}^{n}: [0,T]\to \R$ ,  $k=1,\dots, n$,  
and a function $u_{n}:\left[0,T\right]\to L^2(\Gamma)$ of the form
\begin{equation}
u_{n}\left(t\right)=\sum_{k=1}^{n}y_{k}^{n}\left(t\right)\mathsf{f}_{k},\label{eq: formula of u_m}
\end{equation}
such that for all $k=1,\dots, n$,
\begin{equation}
y_{k}^{n}\left(T\right)=\int_{\Gamma}u_{T}\mathsf{f}_{k}\psi^{2}dx,\label{eq: aproximate terminal condtion u_m}
\end{equation}
and 
\begin{equation}
   -\frac d {dt} ( u_n, \mathsf{f}_{k}\psi )_{L^2(\Gamma)}
  + \int_{\Gamma}\Bigl(\mu\partial u_{n} -H\left(x,u_{n}\right)  \Bigr)   \partial\left(\mathsf{f}_{k}\psi\right) dx
=
-\int_{\Gamma}f \partial (\mathsf{f}_{k}\psi) dx.\label{eq: approximate linear equation for u_m}
\end{equation}
\end{lem}

\begin{proof}[Proof of Lemma~\ref{lem: Galerkin's method for regularity HJ}]
The proof follows the same lines as the one of Lemma~\ref{lem: Galerkin's method}
but it is more technical since we obtain a  system of nonlinear differential equations.
For $n\ge 1$, we consider the symmetric $n$ by $n$ matrix $M_{n}$
defined by 
\[
\left(M_{n}\right)_{k\ell}=\int_{\Gamma}\mathsf{f}_{k}\mathsf{f}_{\ell}\psi dx.
\]
Since $\psi$ is positive and bounded and since $ ( \psi \mathsf{f}_{k}) _{k=1}^{\infty}$
is a Hilbert basis of $L^{2}\left(\Gamma\right)$, we can check
that $M_n$ is a positive definite matrix and there exist two constants
$c,C$ independent of $n$ such that
\begin{equation}
  c\left|\xi\right|^{2}\le\sum_{k,\ell=1}^{n}\left(M_{n}\right)_{k\ell}\xi_{k}\xi_{\ell}\le C\left|\xi\right|^{2},
  \quad\text{for all }\xi\in\R^n.
\end{equation}
Looking for $u_n$ of the form~(\ref{eq: formula of u_m})
and setting $Y=\left(y_{1}^{n},\ldots,y_{n}^{n}\right)^{T}$,
$\dot{Y}=\left( \frac d {dt} y_{1}^{n},\ldots,\frac d {dt}y_{n}^{n}\right)^T$,
\eqref{eq: approximate linear equation} implies that we have to solve the following
a system of ODEs:
\begin{eqnarray}
  \left\{
  \begin{array}{l}
\ds  -M_n \dot{Y}(t)+B Y(t)+\mathcal{H}(Y)(t)
  = G(t),\quad\quad t\in [0,T]\\
\ds  Y(T)= \left( \int_{\Gamma}u_{T}\mathsf{f}_{1}\psi^{2}dx, \cdots ,  \int_{\Gamma}u_{T}\mathsf{f}_{n}\psi^{2}dx\right)^T,
  \end{array}
  \right.
  \label{ode-compliquee}
\end{eqnarray}
where
\begin{itemize}
\item $B_{k\ell}= \int_{\Gamma}  \mu \partial \mathsf{f}_{\ell} \partial (\psi\mathsf{f}_{k}) dx$
\item $\mathcal{H}_i(Y)= - \int_{\Gamma} {H}(x,Y^T F) \partial (\mathsf{f}_i\psi)dx$
with  $F=(\mathsf{f}_{1}, \cdots , \mathsf{f}_{n})^T$ and  $Y^TF = \sum_\ell y_\ell^n \mathsf{f}_{\ell}= u_n$
\item $G_i(t)=-\int_{\Gamma}f(x,t) \partial (\mathsf{f}_{i}\psi) dx$ for all $i\in 1,\cdots, n$.
\end{itemize}
 Since the matrix $M$ is invertible and
the function $\mathcal{H}$ is Lipschitz continuous by (\ref{eq:18}),
the system~\eqref{ode-compliquee} has a unique global solution. This ends the proof of the lemma.
\end{proof}

We start by giving some  estimates for the approximation $u_n$.

\begin{lem}
\label{lem: enegy estimate of u_m in H^2}
Under the assumptions made in Theorem \ref{regularite-sys-derive},
there exists a constant
$C$ depending only on $\Gamma$, $T$, $\psi$, $\left\Vert u_{T}\right\Vert _{F}$,  $\left\Vert \partial f\right\Vert _{L^2\left(\Gamma\times(0,T)\right)}$
 $\|f\|_{C(\Gamma\times[0,T])}$ 
and $\|\partial_t f \|_{L^2(0,T; H^1_b(\Gamma))  }$  
such that 
\[
\left\Vert u_{n}\right\Vert _{L^{\infty}(0,T;F)} 
+ \left\Vert u_{n}\right\Vert _{L^{2}\left(0,T;H_b^{2}\left(\Gamma\right)\right)}
+ \left\Vert \partial_{t}u_{n}\right\Vert _{L^2\left(\Gamma\times(0,T)\right)}
\le 
C.
\]
\end{lem}

\begin{proof}[Proof of Lemma~\ref{lem: enegy estimate of u_m in H^2}]
We divide the proof into two steps:

\emph{Step 1: Uniform estimates of $u_n$ in $L^{\infty}(0,T;L^2(\Gamma))$, $L^{2}(0,T;F)$ and $W^{1,2}(0,T; E')$.}
Multiplying \eqref{eq: approximate linear equation for u_m}
by $y_{k}^{n}\left(t\right)\mathsf{f}_{k}e^{\lambda t}\psi$
where $\lambda$ is a positive constant to be chosen later, summing for $k= 1,\dots, n$
and using  \eqref{eq: formula of u_m}, we get
\begin{displaymath}
 -\int_{\Gamma}\partial_{t}u_{n}u_{n}e^{\lambda t}\psi dx+\int_{\Gamma}\Bigl(\mu\partial u_{n} -H(x,u_n)\Bigr) 
\partial\left(u_{n}e^{\lambda t}\psi\right)dx = -\int_{\Gamma}f \partial(u_{n}\psi e^{\lambda t}) dx.  
\end{displaymath}

In the following lines, $C$ will be a constant that may vary from lines to lines.
Since $H$ satisfies (\ref{eq:17})
 and
$f$ is bounded, there exists a constant $C$  such that 
\begin{align}
 & -    \int_{\Gamma}\left[\partial_{t}\left(\dfrac{u_{n}^{2}}{2}e^{\lambda t}\right)
   -    \dfrac{\lambda}{2}u_{n}^{2}e^{\lambda t}\right]\psi dx+\int_{\Gamma}\mu\left|\partial u_{n}\right|^{2}e^{\lambda t}\psi dx
   -    C \int_{\Gamma} \left|u_{n}\right|\left(|u_n|+ \left|\partial u_{n}\right|\right)e^{\lambda t} dx\nonumber \\
    \le & C \int_{\Gamma} (|u_{n}| + |\partial u_n|)e^{\lambda t} dx.\label{energy estimate regularity HJ_1}
\end{align}
 The desired estimate on $u_n$ is obtained from the previous inequality in a
 similar way as in the  proof of Lemma~\ref{lem: Energy estimate},  by taking $\lambda$ large enough.


\emph{Step 2: Uniform estimates of $u_n$ in $L^{\infty}(0,T;F) \cap L^{2}(0,T;H^2_b(\Gamma))$
and of $\partial_t u_n$ in $L^2 (\Gamma\times(0,T))$.}
Multiplying \eqref{eq: approximate linear equation for u_m}
by $\partial_{t}y_{k}^{n}\left(t\right)\mathsf{f}_{k}e^{\lambda t}\psi$
where $\lambda$ is a positive constant to be chosen later,
integrating by part the term containing $H$ and $f$ (all the integration by parts are justified)
 summing for $k=1,\dots,n$
and using  \eqref{eq: formula of u_m}, we obtain that
\begin{align}\label{eq:33}
 & -\int_{\Gamma} (\partial_{t}u_{n}) ^2 e^{\lambda t}\psi dx+
\int_{\Gamma}\mu\partial u_{n}\partial\left(\partial_t u_{n}e^{\lambda t}\psi\right)dx
 + \int_{\Gamma}\partial \left(H\left(x,u_{n}\right)\right) \partial_ t u_{n}e^{\lambda t}\psi dx\\
 & - \sum_{i\in I} \sum_{\alpha\in\mathcal{A}_{i}} n_{i\alpha} \left[H^\alpha(\nu_{i},u_{n}|_{\Gamma_{\alpha}} \left(\nu_{i},t\right))
   - f|_{\Gamma_{\alpha}}\left(\nu_{i},t\right)\right] \partial_t u_{n}|_{\Gamma_{\alpha}}\left(\nu_{i},t\right)\psi|_{\Gamma_{\alpha}}\left(\nu_{i}\right)e^{\lambda t}=\int_{\Gamma} \partial f \partial_t u_{n}\psi e^{\lambda t} dx. \nonumber
\end{align}
Note that from (\ref{eq:18}) and (\ref{eq:19}),
\begin{equation}
  \label{eq:34}
\left|\partial \left(H\left(x,u_{n}\right)\right)\right|\le C_0 (1+ |u_n|+ |\partial u_n|)
\end{equation}
so, from Step 1, this function is bounded in $L^2(\Gamma\times (0,T))$ by a constant. 
Moreover, 
\begin{equation}
  \label{eq:35}
\int_s^T \int_{\Gamma} \partial f \partial_t u_{n}\psi e^{\lambda t} dx dt\le
C \left(\int_s^T \int_{\Gamma} (\partial f)^2 e^{\lambda t} dx dt \right)^{\frac 1 2} \left(\int_s^T \int_{\Gamma}  (\partial_t u_{n}) ^2  e^{\lambda t} \psi  dx dt \right)^{\frac 1 2},
\end{equation}
and we can also estimate the term $\int_{\Gamma}\mu\partial u_{n}\partial\left(\partial_t u_{n}e^{\lambda t}\psi\right)dx$ as in the proof of Theorem~\ref{thm: regularity for v}. 
 Therefore, the only  new difficulty with respect to the proof of Theorem~\ref{thm: regularity for v} consists of obtaining 
a  bound for  the term
\[
\sum_{i\in I}\sum_{\alpha\in\mathcal{A}_{i}}n_{i\alpha}\left[H^\alpha\left(\nu_{i},u_{n}|_{\Gamma_{\alpha}}\left(\nu_{i},t\right)\right)
   -  f|_{\Gamma_{\alpha}}\left(\nu_{i},t\right)\right]\partial_{t}u_{n}|_{\Gamma_{\alpha}}\left(\nu_{i},t\right) e^{\lambda t} \psi|_{\Gamma_{\alpha}} \left(\nu_{i}\right).
\]
Let  $\mathcal{J}_{i\alpha}(p)$ be the primitive function of $p \mapsto H^\alpha(\nu_i,p)$ such that   $\mathcal{J}_{i\alpha}(0)=0$:
\begin{eqnarray*}
  H^\alpha\left(\nu_{i},u_{n}|_{\Gamma_{\alpha}}\left(\nu_{i},s\right)\right)\partial_{t}u_{n}|_{\Gamma_{\alpha}}\left(\nu_{i},s\right)
  =  \frac d {dt} \mathcal{J}_{i\alpha}(u_{n}|_{\Gamma_{\alpha}}\left(\nu_{i},s\right)).
\end{eqnarray*}
We can then write
\begin{displaymath}
  \begin{split} &- \int_s ^T
\left(  n_{i\alpha}  H^\alpha\left(\nu_{i},u_{n}|_{\Gamma_{\alpha}}\left(\nu_{i},t\right)\right)
  \partial_{t}u_{n}|_{\Gamma_{\alpha}}\left(\nu_{i},t\right) e^{\lambda t} \psi|_{\Gamma_{\alpha}} \left(\nu_{i}\right)\right) dt
\\ =&  n_{i\alpha} \psi|_{\Gamma_{\alpha}} \left(\nu_{i}\right) \left( - \mathcal{J}_{i\alpha}\left(u_{n}|_{\Gamma_{\alpha}} \left(\nu_{i},T\right)\right) e^{\lambda T}
+\mathcal{J}_{i\alpha}\left(u_{n}|_{\Gamma_{\alpha}} \left(\nu_{i},s\right)\right) e^{\lambda s} +
\lambda \int_s ^T
  \mathcal{J}_{i\alpha}\left( u_{n}|_{\Gamma_{\alpha}}\left(\nu_{i},t\right)\right)  e^{\lambda t} dt \right).
  \end{split}
\end{displaymath}
Since $H^\alpha(x, \cdot)$ is sublinear, see (\ref{eq:17}), $|\mathcal{J}_{i\alpha}(p)|$ is subquadratic, i.e., 
$|\mathcal{J}_{i\alpha}(p)|\leq C(1+p^2)$, for a constant $C$ independent of $\alpha$ and $i$. This implies that
\begin{displaymath}
  \begin{split} &\left| \int_s ^T
\left(  n_{i\alpha}  H^\alpha\left(\nu_{i},u_{n}|_{\Gamma_{\alpha}}\left(\nu_{i},t\right)\right)
  \partial_{t}u_{n}|_{\Gamma_{\alpha}}\left(\nu_{i},t\right) e^{\lambda t} \psi|_{\Gamma_{\alpha}} \left(\nu_{i}\right)\right) dt\right|
\\
\le & C\left( e^{\lambda T}+  u_{n}^2 |_{\Gamma_{\alpha}}\left(\nu_{i},T\right) e^{\lambda T} 
 + u_{n}^2 |_{\Gamma_{\alpha}}\left(\nu_{i},s\right) e^{\lambda s}\right) +
C \lambda \int_0 ^T \left(1+ u_{n}^2 |_{\Gamma_{\alpha}}\left(\nu_{i},t\right)\right)  e^{\lambda t} dt .
\end{split}
\end{displaymath}
Note that, from Step 1 and the stability of the trace, $ \lambda \int_s ^T \left(1+ u_{n}^2 |_{\Gamma_{\alpha}}\left(\nu_{i},t\right)\right)  e^{\lambda t} dt \le C \lambda e^{\lambda T}$.
To summarize 
\begin{equation}
  \label{eq:36}
 \begin{split} &\left| \int_s ^T
\left(  n_{i\alpha}  H^\alpha\left(\nu_{i},u_{n}|_{\Gamma_{\alpha}}\left(\nu_{i},t\right)\right)
  \partial_{t}u_{n}|_{\Gamma_{\alpha}}\left(\nu_{i},t\right) e^{\lambda t} \psi|_{\Gamma_{\alpha}} \left(\nu_{i}\right)\right) dt\right|
\\
\le & C\left(   u_{n}^2 |_{\Gamma_{\alpha}}\left(\nu_{i},T\right) e^{\lambda T} 
 + u_{n}^2 |_{\Gamma_{\alpha}}\left(\nu_{i},s\right)^2 e^{\lambda s}\right) + \tilde C(\lambda).
\end{split}
\end{equation}
Similarly, using  the fact that $f\in C(\Gamma\times [0,T])$
and $\partial_t f |_{\Gamma_{\alpha}} \left(\nu_{i},\cdot \right)\in L^2\left(0,T\right)$, and integrating by part, we see that
\begin{eqnarray*}
& &\left|  \int_s^T f|_{\Gamma_{\alpha}}\left(\nu_{i},t\right)\partial_{t}u_{n}|_{\Gamma_{\alpha}}\left(\nu_{i},t\right) e ^{\lambda t} dt\right|\\
  &=& \left| (f|_{\Gamma_{\alpha}} u_{n})|_{\Gamma_{\alpha}}\left(\nu_{i},T\right) e^{\lambda T}
 - (f|_{\Gamma_{\alpha}} u_{n})|_{\Gamma_{\alpha}}\left(\nu_{i},t\right) e^{\lambda s}
  - \int_s^T \left(\lambda f|_{\Gamma_{\alpha}}\left(\nu_{i},t\right) + \partial_{t}f|_{\Gamma_{\alpha}}\left(\nu_{i},t\right)\right) u_{n}|_{\Gamma_{\alpha}}\left(\nu_{i},t\right) e^{\lambda t} dt\right|\\
  &\leq&
  C\left( |u_{n}|_{\Gamma_{\alpha}}\left(\nu_{i},T\right)| e^{\lambda T} +|u_{n}|_{\Gamma_{\alpha}}\left(\nu_{i},s\right)|e^{\lambda s}+
\lambda   \int_s^T \left|u_{n}|_{\Gamma_{\alpha}}\left(\nu_{i},t\right)\right| e^{\lambda t} dt  \right)\\ &&
+ \frac  1 2 \int_s^T  u^2_{n}|_{\Gamma_{\alpha}}\left(\nu_{i},t\right) e^{\lambda t} dt +
 \frac  1 2 \int_s^T  \left(\partial_{t}f|_{\Gamma_{\alpha}}\left(\nu_{i},t\right)\right)^2  e^{\lambda t} dt 
.
\end{eqnarray*}
From Step 1 and the assumptions on $f$, the last three terms in the right hand side of the latter estimate are bounded by a constant depending on $\lambda$, but not on $n$.
 To summarize,
 \begin{equation}
   \label{eq:37}
\left|  \int_s^T f|_{\Gamma_{\alpha}}\left(\nu_{i},t\right)\partial_{t}u_{n}|_{\Gamma_{\alpha}}\left(\nu_{i},t\right) e ^{\lambda t} dt\right|\le C\left( 
 |u_{n}|_{\Gamma_{\alpha}}\left(\nu_{i},T\right)| e^{\lambda T} +|u_{n}|_{\Gamma_{\alpha}}\left(\nu_{i},s\right)|e^{\lambda s}
\right)
  + \tilde C(\lambda).
 \end{equation}
To conclude from (\ref{eq:36}) and (\ref{eq:37}), we use the following estimates
\begin{eqnarray}
  \left\{
  \begin{array}{l}
\ds  \left|u_{n}|_{\Gamma_{\alpha}} \left(\nu_{i},t\right)\right|\le C \left( \int_{\Gamma_{\alpha}}\left|u_{n}\left(x,t\right)\right|dx+             
     \int_{\Gamma_{\alpha}}\left|\partial u_{n}\left(x,t\right)\right|dx \right),\\
\ds  u_{n}^{2}|_{\Gamma_{\alpha}}\left(\nu_{i},t\right)\le C \left( \int_{\Gamma_{\alpha}}u_{n}^{2}\left(x,t\right)dx+\int_{\Gamma_{\alpha}}\left|u_{n}\partial u_{n}\left(x,t\right)\right|dx \right),
  \end{array}
  \right.
  \label{pointwise-estim}
\end{eqnarray}
for $t=s$ and $t=T$.

Then proceeding as in the proof of  Theorem~\ref{thm: regularity for v} and combining (\ref{eq:33}), (\ref{eq:34}), (\ref{eq:35}), (\ref{eq:36}) and (\ref{eq:37}) with 
(\ref{pointwise-estim}), we find the desired estimates by taking $\lambda$ large enough.

Let us end the proof by proving (\ref{pointwise-estim}).
The function $\phi= u_{n}|_{\Gamma_{\alpha}} \left(\cdot,t\right)$ is in $H^1(\Gamma_\alpha)$.
By the continuous embedding $H^1(\Gamma_\alpha) \hookrightarrow C(\Gamma_\alpha)$,
we can define $\phi$ in the pointwise sense (and even at two endpoints of any edges, see~\eqref{eq: v at the vertices}).
For all $\alpha\in \mathcal{A}$ and $x,y\in \Gamma_{\alpha}$,  we have
$
\phi(x) = \phi(y) + \int_{[y,x]} \partial \phi(\xi) d\xi.
$
It follows
\begin{eqnarray*}
  |\Gamma_\alpha| \phi(x)=\int_{\Gamma_{\alpha}} \phi(x)dy
  = \int_{\Gamma_{\alpha}}\phi(y)dy + \int_{\Gamma_{\alpha}} \int_{[y,x]} \partial \phi(\xi) d\xi dy
  \leq \int_{\Gamma_{\alpha}} |\phi(\xi )|d\xi + |\Gamma_\alpha| \int_{\Gamma_{\alpha}} |\partial \phi(\xi)| d\xi,
\end{eqnarray*}
which gives the first estimate setting $x=\nu_i$. The second estimate is obtained in the same way replacing $\phi$
by $\phi^2$ and using the fact that $W^{1,1}(\Gamma_\alpha)ss$ is continuously imbedded in $C(\Gamma_\alpha)$.
\end{proof}



\begin{proof}[Proof of Theorem~\ref{regularite-sys-derive}]
From Lemma \ref{lem: enegy estimate of u_m in H^2}, up to the extraction
of a subsequence, there exists $u\in L^{2}\left(0,T;H_{b}^{2}\left(\Gamma\right)\right)\cap W^{1,2}\left(\Gamma\times(0,T)\right)$
such that
\begin{equation}
\begin{cases}
u_{n}\rightharpoonup u, & \text{in }L^{2}\left(0,T;F\cap H_{b}^{2}\left(\Gamma\right)\right),\\
\partial_{t}u_{n}\rightharpoonup\partial_{t}u, & \text{in }L^2\left(\Gamma\times(0,T)\right).
\end{cases}\label{eq: u_n weakly converges}
\end{equation}
Moreover, by Aubin-Lions Theorem (see Lemma~\ref{aubin-lions-lem}),
\[L^{2}\left(0,T; F\cap H^{2}_b\left(\Gamma\right)\right) \cap W^{1,2}\left(0,T;L^{2}\left(\Gamma\right)\right)
\mathop{\hookrightarrow}^{\tiny\rm compact} L^{2}\left(0,T;F\right),\]
so up to the extraction of a subsequence, we may assume that $u_{n}\to u$ in $L^2 (0,T;F)$ and almost everywhere.
Moreover, from the compactness of the trace operator from  $W^{1,2}(\Gamma_\alpha \times (0,T))$ to 
$ L^2(\partial \Gamma_\alpha \times (0,T) )$, $u_n|_{\partial \Gamma_\alpha \times (0,T)} \to u|_{\partial \Gamma_\alpha \times (0,T)}$ in $ L^2(\partial \Gamma_\alpha \times (0,T) )$ and for almost every $t\in (0,T)$.
 Similarly, 
$u_n|_{ \Gamma_\alpha \times \{t=T\}} \to u|_{ \Gamma_\alpha \times \{t=T\}}$ 
in $ L^2( \Gamma_\alpha  )$ and almost everywhere in $\Gamma_\alpha$.
Then, using the Lipschitz continuity of $H$ with respect to its second argument, and similar arguments as in the proof of Theorem \ref{thm: existence and uniqueness linear equation},
we obtain the existence of a solution of~\eqref{eq: regularity for H-J} satisfying~\eqref{energ-est}
by letting $n\rightarrow +\infty$.
Since $H^2(\Gamma_\alpha)\subset C^{1+\sigma}(\Gamma_\alpha)$ for some
$\sigma\in (0,1/2)$,  $u(\cdot, t) \in C^{1+\sigma}\left(\Gamma_\alpha\right)$ for all $\alpha\in\mathcal{A}$ and a.a. $t$.

Finally, the proof of uniqueness is a consequence of the energy estimate~\eqref{energ-est} for $u$.
\end{proof}

Next, we want to prove that, if $u$ is the solution of~\eqref{eq: regularity for H-J}
and $v$ is the solution~\eqref{eq: H-J}, then $\partial u=v$. It means that we have to define
a primitive function on the network $\Gamma$.

\begin{defn} \label{def-chemin}
Let $x\in\Gamma_{\alpha_{0}}=\left[\nu_{i_{0}},\nu_{i_{1}}\right]$
and $y\in\Gamma_{\alpha_{m}}=\left[\nu_{i_{m}},\nu_{i_{m+1}}\right]$.
We denote the set of paths joining from $x$ to $y$ by $\overrightarrow{xy}$.
More precisely, if $\mathcal{L}\in\overrightarrow{xy}$, we can
write $\mathcal{L}$ under the form
\[
\mathcal{L}=x\rightarrow\nu_{i_{1}}\rightarrow\nu_{i_{2}}\rightarrow\ldots\rightarrow\nu_{i_{m}}\rightarrow y,
\]
with $\nu_{i_{k}}\in\mathcal{V}$ and $\left[\nu_{i_{k}},\nu_{i_{k+1}}\right]=\Gamma_{\alpha_{k}}$.
The integral of a function $\phi$ on $\mathcal{L}$ is defined by
\begin{equation}
\int_{\mathcal{L}}\phi\left(\xi\right)d\xi 
=  \int_{[x,\nu_{i_1}]} \phi\left(\xi\right)d\xi
  + \sum_{k=1}^{m} \int_{[\nu_{i_{k}},\nu_{i_{k+1}}]} \phi\left(\xi\right) d\xi
  +  \int_{[\nu_{i_{m}},y]} \phi\left(\xi\right) d\xi,
\end{equation}
recalling that the integrals on a segment are defined in \eqref{int-segment}.
\end{defn}

\begin{lem}\label{cor: well-defined of u}
Let $u$ be the unique solution of \eqref{eq: regularity for H-J} with $u_{T}=\partial v_{T}$. Then for all $x,y\in\Gamma$
and a.e. $t\in\left[0,T\right]$,
\[
\int_{\mathcal{L}_{1}}u\left(\zeta,t\right)d\zeta
=\int_{\mathcal{L}_{2}}u\left(\zeta,t\right)d\zeta,\quad\text{for all }\mathcal{L}_{1},\mathcal{L}_{2}\in\overrightarrow{xy}.
\]
This means that the integral of $u$ from $x$ to $y$ does not depend on the path. Hence,
for any $\mathcal{L}\in\overrightarrow{xy}$, we can define
\[
\int_{\overrightarrow{xy}}u\left(\zeta,t\right)d\zeta:=\int_{\mathcal{L}}u\left(\zeta,t\right)d\zeta.
\]
\end{lem}

\begin{proof}[Proof of Lemma~\ref{cor: well-defined of u}]
First, it is sufficient to prove $\int_{\mathcal{L}}u\left(\zeta,t\right)d\zeta=0$
for all $\mathcal{L}\in \overrightarrow{xx}$. Secondly, if a given edge is browsed twice in opposite senses,
the two related contributions to the integral sum to zero. It follows that, without
loss of generality, we only need to consider loops in $\overrightarrow{xx}$  such that all the complete edges
 that it contains are browsed once only. It is also easy to see that we can focus on the case when  $x\in \cV$.
To summarize, we only need to prove that
\[
\int_{\mathcal{L}}u\left(\zeta,t\right)d\zeta=0
\]
when $\nu_{i_{0}}\in\mathcal{V}\backslash\partial\Gamma$ and
$\mathcal{L}=\nu_{i_{0}}\rightarrow\nu_{i_{1}}\rightarrow\ldots\rightarrow\nu_{i_{m}}\rightarrow\nu_{i_{0}}$,
where $\nu_{i_{k}}\ne\nu_{i_{\ell}}$ for $k\not= l$.

The following conditions 
\begin{enumerate}
\item $ \mathsf{e}|_{\Gamma_\alpha}=0$  on each edge $\Gamma_\alpha$ not contained in $\cL$
\item for all $k=0,\dots m-1$,  $ \mathsf{e}|_{\Gamma_{\alpha_k}}= 1_{i_k<i_{k+1} }  - 1_{i_k>i_{k+1} }$ if
 $\Gamma_{\alpha_k}$ is the edge joining $\nu_{i_k}$ and  $\nu_{i_{k+1}}$
\item $ \mathsf{e}|_{\Gamma_{\alpha_m}}= 1_{i_m<i_{0} }  - 1_{i_m>i_{0} }$ if 
 $\Gamma_{\alpha_m}$ is the edge joining $\nu_{i_m}$ and  $\nu_{i_{0}}$
\end{enumerate}
define a unique function $ \mathsf{e}\in E$ which takes at most two values  on $\cL$, namely $\pm 1$.

From Definition~\ref{def-chemin}, we have
\begin{eqnarray*}
  \frac {d} {dt}\int_{\mathcal{L}}u\left(\zeta,t\right)dt
  &=&\sum_{k=0}^{m} \frac {d} {dt}\int_{[\nu_{i_{k}},\nu_{i_{k+1}}]}u\left(\zeta,t\right)d\zeta+ \frac {d} {dt}\int_{[\nu_{i_{m}},\nu_{i_{0}}]}u\left(\zeta,t\right)d\zeta\\
  &=&  \frac {d} {dt}\int_{\Gamma} u\left(\zeta,t\right)\mathsf{e}\left(\zeta\right)d\zeta 
  =\int_{\Gamma}\partial_{t}u\left(\zeta,t\right)\mathsf{e}\left(\zeta\right)d\zeta.
\end{eqnarray*}
Then, using  Definition~\ref{def882}, Remark \ref{sec:regul-hamilt-jacobi-1} and Remark~\ref{explication-def-weak-sol} yields that
\begin{eqnarray*}
  && \frac d{dt}\int_{\mathcal{L}}u\left(\zeta,t\right)d\zeta\\
&=& \sum_{\alpha\in\mathcal{A}} \int_{\Gamma_{\alpha}}\left[-\mu_{\alpha}\partial^{2}u\left(\zeta,t\right)+\partial H\left(\zeta,u\left(\zeta,t\right)\right)-\partial f\left(\zeta,t\right)\right]\mathsf{e}\left(\zeta\right)d\zeta\\
  &= & \sum_{k=0}^{m} \int_{\Gamma_{\alpha_{k}}}\left[-\mu_{\alpha_{k}}\partial^{2}u\left(\zeta,t\right)+\partial H\left(\zeta,u\left(\zeta,t\right)\right)-\partial f\left(\zeta,t\right)\right]\mathsf{e}\left(\zeta\right) d\zeta\\
  &= & \sum_{k=0}^{m} \mathsf{e}|_{\Gamma_{\alpha_k}} (\nu_i)\left(
    \begin{array}[c]{l}
      n_{i_{k+1}\alpha_{k}}  
 \left( -\mu_{\alpha_{k}}\partial u|_{\Gamma_{\alpha_k}}\left(\nu_{i_{k+1}},t\right) + H^\alpha\left(\nu_{i_{k+1}},u|_{\Gamma_{\alpha_k}}\left(\nu_{i_{k+1}},t\right)\right) 
- f\left(\nu_{i_{k+1}},t\right)\right)\\
 +n_{i_k\alpha_{k}}  \left( -\mu_{\alpha_{k}}\partial u |_{\Gamma_{\alpha_k}}\left(\nu_{i_{k}},t\right) + H^\alpha\left(\nu_{i_{k}},u|_{\Gamma_{\alpha_k}}\left(\nu_{i_{k}},t\right)\right) 
- f\left(\nu_{i_{k}},t\right)\right)
    \end{array}
\right)         ,
\end{eqnarray*}
where we have set $i_{m+1}=i_0$. Now using (\ref{eq:25}) (which is satisfied for a.e. $t$ from the regularity of $u$) and the fact that $ \mathsf{e}\in E$, we conclude that
\begin{equation}
  \label{eq:38}
  \frac d{dt}\int_{\mathcal{L}}u\left(\zeta,t\right)d\zeta=0.
\end{equation}

Hence
\[
\int_{\mathcal{L}}u\left(\zeta,t\right){d\zeta}=\int_{\mathcal{L}}u\left(\zeta,T\right){d\zeta}
= \int_{\mathcal{L}}u_T\left(\zeta\right){d\zeta}
=\int_{\mathcal{L}}\partial v_T\left(\zeta\right){d\zeta}=0,
\]
where the last identity comes from the assumption that $v_T\in V$ (the continuity of $v_T$).

\end{proof}


\begin{lem}
\label{thm: recover v}
If $u_{T}=\partial v_{T}\in F$, then the weak solution $u$ of~\eqref{eq: regularity for H-J} satisfies
$u=\partial v$ where $v$ is the unique solution of \eqref{eq: H-J}.
\end{lem}

\begin{proof}[Proof of Lemma~\ref{thm: recover v}]
For simplicity, we write the proof in the case when $\partial \Gamma\not= \emptyset$. The proof is similar in the other case.

Let us fix some vertex $\nu_k\in \partial \Gamma$.
From standard  regularity results for Hamilton-Jacobi equation with homogeneous Neumann condition, we know that 
 that there exists $\omega$, a closed neighborhood of $\{\nu_k\}$ in $\Gamma$ made of a single straigt line segment and 
 containing no other vertices of $\Gamma$ than $\nu_k$, 
 such that $v|_{\omega \times (0,T)}\in L^2(0,T;H^3(\omega))\cap C([0,T]; H^2(\omega)\cap W^{1,2}(0,T;H^1(\omega))$.
Hence, $v$ satisfies the Hamilton-Jacobi equation at almost every point of $\omega\times(0,T)$. Moreover the equation 
\begin{equation}
  \label{eq:39}
\partial_t v (\nu_k, t)+\mu \partial ^2 v (\nu_k, t) -H(\nu_k, 0)+f(\nu_k, t)=0
\end{equation}
 holds for almost every $t\in (0,T)$ and in $L^2(0,T)$.

For every $x\in\Gamma$ and $t\in [0,T]$, we define
\begin{equation}
 \hat{v}\left(x,t\right)=v\left(\nu_k,t\right)
+\int_{\overrightarrow{\nu_k x}}u\left(\zeta,t\right)d\zeta.\label{eq: v_hat}
\end{equation}

 \begin{rem}
   \label{sec:regul-hamilt-jacobi-2}
If $\partial \Gamma=\emptyset$, then the proof should be modified by replacing $\nu_k$ by a point $\nu\in \Gamma \setminus \cV$ and 
by using local regularity results for  the HJB equation in (\ref{eq: H-J}). 
 \end{rem}

We claim that $\hat{v}$ is a solution of \eqref{eq: H-J}.

First, $\hat{v}\left(\cdot,t\right)$ is continuous on $\Gamma$.
Indeed,  $\hat{v}(y,t)-\hat{v}(x,t)= \int_{\overrightarrow{xy}} u(\zeta,t)d\zeta$.
On the other hand,  $u\in C([0,T]; F)
\subset L^\infty(\Gamma\times [0,T])$.
It follows that $|\hat{v}(y,t)-\hat{v}(x,t)|\leq ||u||_{L^\infty(\Gamma\times [0,T])} {\rm dist}(x,y)$
which implies  that  $\hat{v}\left(\cdot,t\right)$ is continuous on $\Gamma$.

Next, from the terminal conditions for $u$,
\[
\hat{v}\left(x,T\right)=v\left(\nu_k,T\right)+\int_{\overrightarrow{\nu_k x}}u\left(\zeta,T\right)d\zeta=v_T\left(\nu_k\right)+\int_{\overrightarrow{\nu_k x}}\partial v_{T}\left(\zeta\right)d\zeta=v_{T}\left(x\right),
\]
where the last identity follows from the continuity of $v_{T}$ on $\Gamma$.

Let us check the Kirchhoff condition for $\hat{v}$. Take $\nu_{i}\in\mathcal{V}$ and $\alpha\in \mathcal{A}_i$.
From~\eqref{eq:2}, for a.e. $t\in (0,T)$,
$\partial_\alpha \hat{v} (\nu_i,t)= n_{i\alpha} \partial \hat{v}|_{\Gamma_{\alpha}}(\nu_i,t)$ and from (\ref{eq: v_hat}),
$\partial \hat{v}|_{\Gamma_{\alpha}}(\nu_i,t)=  u|_{\Gamma_{\alpha}}(\nu_i,t)$.
Since $u(\cdot, t)\in F$, we get
\begin{eqnarray*}
  \sum_{\alpha\in \mathcal{A}_i} \gamma_{i\alpha}\mu_\alpha \partial_\alpha \hat{v} (\nu_i,t)
=\sum_{\alpha\in \mathcal{A}_i} \gamma_{i\alpha}\mu_\alpha  n_{i\alpha} u|_{\Gamma_\alpha} (\nu_i,t)  =0,
\end{eqnarray*}
which is exactly the  Kirchhoff condition for $\hat{v}$ at $\nu_i$.

There remains to prove $\hat v$ solves the Hamilton-Jacobi equation in $\Gamma\setminus \cV$:
Take  $x\in\Gamma_{\alpha}\backslash\mathcal{V}$ for some $\alpha\in \cA$ and consider a path
$\overrightarrow{\nu_k x} \owns \mathcal{L}=\nu_{i_0} \rightarrow\cdots \rightarrow \nu_{i_m} \rightarrow x$,
where $i_0=k$ and $\nu_{i_m}\in \Gamma_{\alpha}$.  Let $\nu_{i_{m+1}}$ be the other endpoint of $\Gamma_{\alpha}$.
We proceed as in the proof of Lemma~\ref{cor: well-defined of u}:
the following conditions 
\begin{enumerate}
\item $ \mathsf{e}|_{\Gamma_\alpha}=0$  on each edge $\Gamma_\alpha$ not contained in $\cL$
\item for all $j=0,\dots m$,  $ \mathsf{e}|_{\Gamma_j}= 1_{i_j<i_{j+1} }  - 1_{i_j>i_{j+1} }$ if
 $\Gamma_j$ is the edge joining $\nu_{i_j}$ and  $\nu_{i_{j+1}}$
\end{enumerate}
define a unique piecewise constant function $ \mathsf{e}$ which takes at most two values  on $\cL$, namely $\pm 1$.
Note that $\mathsf{e}$ does not belong to $E$ because  $\mathsf{e}(\nu_k)\not =0$, but that $\mathsf{e}$ satisfies $
\sum_{\alpha\in\mathcal{A}_{i}}  n_{i\alpha}\mathsf{e}|{}_{\Gamma_{\alpha}}\left(\nu_{i}\right)=0\text{ for all }\nu_{i}\in\mathcal{V}\setminus\partial  \Gamma$.

Using this function, a similar computation as in the proof of Lemma~\ref{cor: well-defined of u} implies that, for almost every $t\in (0,T)$,
\begin{displaymath}
  \begin{split}
  \partial_t \hat v (x,t)-\partial_t v (\nu_k,t)
=&
  -\mu_{\alpha}\partial u|_{\Gamma_{\alpha}}\left(x,t\right) + H\left(x,u|_{\Gamma_{\alpha}}\left(x,t\right)\right) - f\left(x,t\right)\\
&+  \mu_{\alpha}\partial_2 v\left(\nu_k,t\right) - H\left(\nu_k,0\right) + f\left(\nu_k,t\right).
  \end{split}
\end{displaymath}
Then, using (\ref{eq:39}) and the fact that $\partial \hat v = u$, the latter identity yields that for almost every $(x,t)\in (0,T)\times\Gamma$,
\begin{displaymath}
  \partial_t \hat v (x,t)
+  \mu_{\alpha}  \partial^2\hat v  \left(x,t\right) - H\left(x,\partial \hat v\left(x,t\right)\right) 
+ f\left(x,t\right)
=0.
\end{displaymath}
We have proven that $\hat v$ is a solution of \eqref{eq: H-J}. Since $v$ is the unique solution of
\eqref{eq: H-J}, we conclude that $v=\hat{v}$ and $\partial v=u$. 
\end{proof}

We are now ready to give the proof of  Theorem~\ref{thm: H-J regularity}.

\begin{proof}[Proof of Theorem~\ref{thm: H-J regularity}]
Since $\partial v =u$ by Lemma~\ref{thm: recover v} and $u$
satisfies~\eqref{energ-est} by Theorem~\ref{regularite-sys-derive},
we obtain that 
$v\in L^{2}\left(0,T;H^{3}\left(\Gamma\right)\right)$ and
$\partial_{t}v\in L^{2}\left(0,T;H^{1}\left(\Gamma\right)\right)$
and~\eqref{energy-est-v} holds.

Therefore, using an interpolation result combined with Sobolev embeddings, see \cite{Amann2000} or Lemma~\ref{amann-result} in the Appendix,
$v\in C^{1+\sigma, \sigma/2}(\Gamma\times [0,T])$ for some $0<\sigma< 1$.

Finally, we know that since $f\in W^{1,2}(0,T,H^1_b(\Gamma) )$, $f|_{\Gamma_\alpha\times [0,T]}\in C^{\eta, \eta}(\Gamma_\alpha\times [0,T])$ for all $\eta\in (0,1/2)$. If  $f \in C^{\eta,\frac{\eta}{2}}(\Gamma_\alpha\times [0,T])$ for some $\eta\in (0, 1/2)$, we claim that $v\in C^{2,1}(\Gamma\times [0,T])$.
 This is a direct consequence of a theorem of  Von Below, see the main theorem in \cite{Below1988},
for the (modified) heat equation 
\begin{equation}
-\partial_t w -\mu_{\alpha}\partial^2 w = g(x,t) \quad \text {in } (\Gamma_{\alpha}\backslash\mathcal{V})\times (0,T),
\end{equation}
with the same Kirchhoff conditions  as in (\ref{eq: H-J}):
Note that if the terminal Cauchy condition for $w$ is $w(\cdot,t=T)=v_T$ and if  $g=f-H(x,\partial v)$,
then $w=v$.  Now $g=f-H(x,\partial v)\in C^{\tau,\frac{\tau}{2}}(\Gamma_\alpha\times [0,T])$, where $1/2>\tau=\min (\sigma, \eta)>0$. Using the result in \cite{Below1988}, we obtain that $v=w\in C^{2+\tau, 1+\tau/2} (\Gamma_\alpha\times [0,T])$, then that $v$  is a classical solution of (\ref{eq: H-J}).
\end{proof}

\section{Existence, uniqueness and regularity for the MFG system
(Proof of Theorem~\ref{thm: MFG system})}
\label{sec:exist-uniq-regul}
\begin{proof}
[Proof of existence in Theorem \ref{thm: MFG system}]
Given $m_0$ and $v_T$, let us construct the map $T:L^{2}\left(0,T;V\right)\rightarrow L^{2}\left(0,T;V\right)$
as follows. 

Given $v\in L^{2}\left(0,T;V\right)$,  we first define $m$ as the weak solution of (\ref{eq: Fokker-Planck})  with initial data $m_0$ and $b= H_p(x, \partial v)$. We know that $ m\in L^{2}\left(0,T;W\right)\cap C([0,T]; L^2(\Gamma))\cap W^{1,2}(0,T; V')$. 

We claim that if $v_n\to v$ in $L^{2}\left(0,T;V\right)$ then 
 $H_p(\cdot, \partial v_n)$ tends to $H_p(\cdot, \partial v)$ in $L^2 (\Gamma\times (0,T))$. To prove the claim, we argue by
 contradiction: assume that
there exist a positive number $\epsilon$ and a subsequence $v_{\phi(n)}$ such that 
$\|H_p(\cdot, \partial v_{\phi(n)})-H_p(\cdot, \partial v)\|_{L^2 (\Gamma\times (0,T))}>\epsilon$. Then since $\partial v_{\phi(n)}$ tends to $\partial v$ in  $L^2 (\Gamma\times (0,T))$, we can extract another subsequence $v_{\psi(n)}$ from $v_{\phi(n)}$ such that $\partial v_{\psi(n)}$ tends to $\partial v$ almost every where in $\Gamma\times (0,T)$. 
From the continuity of $H_p$, we deduce that $H_p(\cdot, \partial v_{\psi(n)})$ tends to $H_p(\cdot, \partial v)$  almost everywhere 
in  $\Gamma\times (0,T)$. 
Since there exists  a positive constant $C_0$ such that $\|H_p(\cdot, \partial v_{\psi(n)})\|_{\infty}\le C_0$,  $\|H_p(\cdot, \partial v)\|_{\infty}\le C_0$, Lebesgue dominated convergence theorem ensures that 
$H_p(\cdot, \partial v_{\psi(n)})$ tends to $H_p(\cdot, \partial v)$ in $L^2(\Gamma\times (0,T))$, which is the desired contradiction.

To summarize,
 $H_p(\cdot, \partial v_n)$ tends to $H_p(\cdot, \partial v)$ in $L^2 (\Gamma\times (0,T))$ on the one hand, and for a positive constant $C_0$, $\|H_p(\cdot, \partial v_n)\|_{\infty}\le C_0$,  $\|H_p(\cdot, \partial v)\|_{\infty}\le C_0$. Using Lemma~\ref{lem: stability FK}, we see that $m_n$,   the weak solution of (\ref{eq: Fokker-Planck})  with initial data $m_0$ and $b= H_p(x, \partial v_n)$
converges to $m$ in $L^{2}\left(0,T;W\right)\cap L^\infty\left(0,T;L^2(\Gamma)\right)\cap W^{1,2}(0,T; V')$.
Hence,
the map $v\mapsto m$ is continuous from $L^{2}\left(0,T;V\right)$ to 
$L^{2}\left(0,T;W\right)\cap L^\infty\left(0,T;L^2(\Gamma)\right)\cap W^{1,2}(0,T; V')$. Moreover, the  a priori estimate
(\ref{eq:22}) holds uniformly with respect to $v$.

Then, knowing $m$, we construct $T(v)\equiv \widetilde v$ as the unique weak solution of (\ref{eq: H-J}) with $f(x,t)= \ccV[m(\cdot,t)](x)$.
Note  that $m\mapsto f$ is continuous and locally bounded from $L^2(\Gamma \times (0,T))$ to
$ L^2(\Gamma\times (0,T))$. Then Lemma \ref{lem: stability HJ} ensures that 
the map $m\to \tilde v$ is continuous from $L^2(\Gamma \times (0,T))$ to
 $L^{2}\left(0,T;H^2(\Gamma)\right)\cap L^\infty\left(0,T;V\right)\cap W^{1,2}(0,T; L^2(\Gamma))$. 
From Aubin-Lions theorem, see Lemma~\ref{aubin-lions-lem},  $m\to \tilde v$ maps bounded sets of  $L^2(\Gamma \times (0,T))$ to relatively compact sets of $L^{2}\left(0,T;V\right)$.

Therefore, the map $T: v\mapsto \tilde v$ is continuous from $L^{2}\left(0,T;V\right)$ to $L^{2}\left(0,T;V\right)$
and has a relatively compact image. Finally, we apply Schauder fixed point theorem~\cite[Corollary 11.2]{GT2001}
and conclude that the map $T$ admits a fixed point $v$. We know that 
$ v\in L^{2}\left(0,T;H^2(\Gamma)\right)\cap L^\infty\left(0,T;V\right)\cap W^{1,2}(0,T; L^2(\Gamma))$
and $m\in L^{2}\left(0,T;W\right)\cap L^\infty\left(0,T;L^2(\Gamma)\right)\cap W^{1,2}(0,T; V'(\Gamma))$.

Hence, there exists a weak solution $\left(v,m\right)$ to the mean
field games system \eqref{eq: MFG system}.
\end{proof}

\begin{proof}
[Proof of uniqueness in Theorem \ref{thm: MFG system}]

We assume that there exist two solutions $\left(v_{1},m_{1}\right)$
and $\left(v_{2},m_{2}\right)$ of \eqref{eq: MFG system}. We set $\overline{v}=v_{1}-v_{2}$
and $\overline{m}=m_{1}-m_{2}$ and write the system for $\overline{v},\overline{m}$
\[
\begin{cases}
 - \partial_{t}\overline{v} - \mu_{\alpha}\partial^{2}\overline{v} 
 + H\left(x,\partial v_{1}\right) - H\left(x,\partial v_{2}\right)
 - \left(\ccV\left[m_{1}\right]-\ccV\left[m_{2}\right]\right)=0, 
 & x \in \Gamma_{\alpha} \backslash \mathcal{V},~t\in\left(0,T\right),\\
   \partial_{t}\overline{m} - \mu_{\alpha}\partial^{2}\overline{m}
 - \partial\left(m_{1}\partial_{p}H\left(x,\partial m_{1}\right) - m_{2}\partial_{p}H\left(x,\partial m_{2}\right)\right)=0 
 & x \in \Gamma_{\alpha} \backslash \mathcal{V},~t\in\left(0,T\right),\\
   \overline{v}|_{\Gamma_{\alpha}} \left(\nu_{i},t\right) = \overline{v}|_{\Gamma_{\beta}}\left(\nu_{i},t\right), \
   \dfrac{\overline{m}|_{\Gamma_{\alpha}}\left(\nu_{i},t\right)}{\gamma_{i\alpha}} = \dfrac{\overline{m}|_{\Gamma_{\beta}}\left(\nu_{i},t\right)}{\gamma_{i\beta}}, 
 & \alpha,\beta\in\mathcal{A}_{i},~\nu_i \in \mathcal{V},\\
   \ds{\sum_{\alpha\in\mathcal{A}_{i}}\gamma_{i\alpha}\mu_{\alpha}\partial_{\alpha}\overline{v}\left(\nu_{i},t\right)=0,} & \nu_i \in \mathcal{V},~t\in\left(0,T\right),\\
   \ds{\sum_{\alpha\in\mathcal{A}_{i}} n_{i\alpha}\left[m_{1}|{}_{\Gamma_{\alpha}}\left(\nu_{i}\right)\partial_{p}H^\alpha\left(\nu_{i},\partial v_{1}|_{\Gamma_{\alpha}}\left(\nu_{i},t\right)\right)-m_{2}|{}_{\Gamma_{\alpha}}\left(\nu_{i}\right)\partial_{p}H^\alpha\left(\nu_{i},\partial v_{2}|_{\Gamma_{\alpha}}\left(\nu_{i},t\right)\right)\right]}\\
 \hspace*{0.4cm}  \ds{+\sum_{\alpha\in\mathcal{A}_{i}}\mu_{\alpha}\partial_{\alpha}\overline{m}\left(\nu_{i},t\right)=0,} & \nu_i\in \mathcal{V},~t\in\left(0,T\right),\\
\overline{v}\left(x,T\right)=0,\ \overline{m}\left(x,0\right)=0.
\end{cases}
\]
Testing by $\overline{m}$ the boundary value problem satified by  $\overline{u}$,
 testing by $\overline{u}$ the boundary value problem satified by  $\overline{m}$,
subtracting,
we obtain
\begin{align*}
 &   \int_{0}^{T}\int_{\Gamma}\left(m_{1} - m_{2}\right)\left(\ccV\left[m_{1}\right] - \ccV\left[m_{2}\right)\right]dxdt
   + \int_{0}^{T}\int_{\Gamma}\partial_{t}\left(\overline{m}~\overline{v}\right)dxdt\\
 & + \sum_{\alpha\in\mathcal{A}}\int_{\Gamma_{\alpha}}m_{1}\left[H\left(x,\partial v_{2}\right) - H\left(x,\partial v_{1}\right)
   - \partial_{p}H\left(x,\partial v_{1}\right)\partial\overline{v}\right]dx\\
 & + \sum_{\alpha\in\mathcal{A}}\int_{\Gamma_{\alpha}}m_{2}\left[H\left(x,\partial v_{1}\right) - H\left(x,\partial v_{2}\right)
   + \partial_{p}H\left(x,\partial v_{1}\right)\partial\overline{v}\right]dx =0.
\end{align*}
Since $V$ is strictly monotone, the first sum is nonnegative.
Moreover,
\[
    \int_{0}^{T}\int_{\Gamma}\partial_{t}\left(\overline{m}~\overline{v}\right)dxdt=\int_{\Gamma} [\overline{m}(x,T)\ \overline{v}(x,T) 
  - \overline{m}(x,0)\ \overline{v}(x,0)] dx
 =  0 ,
 \]
 since $\overline{v}(x,T)=0$ and  $\overline{m}(x,0)=0$.
From the convexity of $H$ and the fact that $m_{1},m_{2}$ are nonnegative,
the last two sums are nonnegative. Therefore, all the terms are zero and thanks
again to the fact that $\ccV$ is strictly increasing, we obtain $m_{1}=m_{2}$.
From Lemma \ref{thm: existence and uniqueness of HJ equation}, we
finally obtain $v_{1}=v_{2}$.
\end{proof}
\begin{proof}[Proof of regularity in Theorem~\ref{thm: MFG system}]
We make the  stronger assumptions written in Section ~\ref{sec:strong-assumpt-coupl} on the coupling operator $\ccV$. We know that $ \ccV[m] \in W^{1,2}(0,T, H^1_b(\Gamma))\cap PC(
\Gamma\times [0,T])$. Assuming also that $v_T\in V$ and $\partial v_T \in F$, we can apply the regularity result in Theorem 
\ref{thm: H-J regularity}: $v\in L^{2}\left(0,T;H^{3}\left(\Gamma\right)\right)\cap  W^{1,2}\left(0,T;H^{1}\left(\Gamma\right)\right)$.

 Moreover, since $ \ccV[m] \in W^{1,2}(0,T, H^1_b(\Gamma))$, we know that 
 $(\ccV[m])|_{\Gamma_ \alpha \times [0,T]} \in C^{\sigma, \sigma/2}(\Gamma_\alpha\times [0,T])$ for all $0<\sigma<1/2$.
If $v_T\in C^{2+\eta}\cap  D$ for some $\eta\in (0,1)$ ($D$ is defined in (\ref{Kirchhoff condition})), then from Theorem 
\ref{thm: H-J regularity}, $v\in C^{2+\tau, 1+\tau/2}(\Gamma\times [0,T])$ for some $\tau\in (0,1)$ and the boundary value problem for $v$ is satisfied in a classical sense.

In turn,  if  for all $\alpha \in \cA$, $\partial p H^\alpha (x, p)$ is a Lipschitz
 function defined in $\Gamma_\alpha \times \R$, and if $m_0\in W$,
 then  we can use  the latter regularity of $v$ and arguments similar to those contained in the proof of Theorem~\ref{thm: regularity for v} and prove that $m\in C([0,T];W)\cap W^{1,2}(0,T; L^2(\Gamma))\cap L^2 (0,T; H^2_b(\Gamma))$.
\end{proof}

\appendix

\section{Some continuous and compact  embeddings}
\label{app:embeddings}

\begin{lem} \label{aubin-lions-lem}(Aubin-Lions Lemma, see \cite{Lions1969})
Let $X_{0}$,$ X$ and $X_{1}$ be function spaces,  ($X_0$ and $X_1$ are reflexive).
Suppose that $X_{0}$ is compactly embedded in $X$ and that $X$ is continuously embedded
in $X_{1}$. Consider some real numbers $1< p,q<+\infty$. Then the following set
\[
\left\{ v:(0,T)\mapsto X_0: \; v\in L^{p}\left(0,T;X_{0}\right),\; \partial_{t}v\in L^{q}\left(0,T;X_{1}\right)\right\}
\]
is compactly embedded
in $L^{p}\left(0,T;X\right)$.
\end{lem}
\begin{lem} \label{amann-result}(Amann, see \cite{Amann2000})
Let $\phi :[a,b]\times [0,T]\to \R$ such that
$\phi \in L^2(0,T;H^2(a,b))$ and $\partial_t\phi \in L^2(0,T;L^2(a,b))$.
Then $\phi\in C^s (0,T;H^1(a,b))$ for some $s\in (0,1/2)$.
\end{lem}

This result is a consequence of the general result~\cite[Theorem 1.1]{Amann2000}
taking into account~\cite[Remark 7.4]{Amann2000}. More precisely, we have
\begin{eqnarray*}
E_1:=H^2(a,b)  \mathop{\hookrightarrow}^{\tiny\rm compact}  E:=H^1(a,b) \hookrightarrow E_0:=L^2(a,b).
\end{eqnarray*}
Let $r_0=r_1=r=2$, $\sigma_0=0$, $\sigma_1=2$ and $\sigma=1$. For any $\nu \in (0,1)$, we define 
\[
\dfrac{1}{r_\nu}=\dfrac{1}{r_0}+\dfrac{1-\nu}{r_1},\quad \sigma_\nu:=(1-\nu)s_0 + \nu s_1.
\]
This implies that $r_\nu=2$ and $\sigma_\nu=2\nu$.
Therefore, if $\nu\in(1/2,1)$, then the following inequality is satisfied
\[
\sigma-1/r<\sigma_\nu - 1/r_\nu<\sigma_1-1/r_1.
\]
Hence, we infer from~\cite[Remark 7.4]{Amann2000}
\[
E_1\hookrightarrow (E_0,E_1)_{\nu,1} \hookrightarrow(E_0.E_1)_{\nu,r_\nu} = W^{\sigma_\nu, r_\nu}(a,b) \hookrightarrow E,
\]
where  $(E_0,E_1)_{\nu,1}$, $(E_0.E_1)_{\nu,r_\nu}$ are interpolation spaces.
This is precisely the assumption allowing to apply~\cite[Theorem 1.1]{Amann2000},
which gives the result of Lemma~\ref{amann-result}.
\bigskip

\noindent {\bf Acknowledgment.
  This work was partially supported by the ANR (Agence Nationale de la Recherche) through
MFG project ANR-16-CE40-0015-01.
The work of O. Ley and N. Tchou is partially supported by the Centre Henri Lebesgue ANR-11-LABX-0020-01.}



\end{document}